\documentclass[reqno, 11pt]{amsart}
\usepackage{dsfont}
\usepackage{amssymb}
\usepackage{mathrsfs}
\usepackage{enumitem}
\usepackage{comment}
\usepackage[pdfstartview=FitH, pdfborder={0 0 0}, colorlinks=true, citecolor=blue, linkcolor=blue]{hyperref}
\usepackage{aliascnt}
\usepackage[backend=biber, style=alphabetic]{biblatex}
\AtBeginBibliography{\setlength{\emergencystretch}{2em}}
\hypersetup{
	pdftitle={Heat kernel geometry and Gromov's volume growth conjecture},
	pdfauthor={Jian Ge},
	pdfsubject={Heat-kernel proof of a uniform volume-growth estimate},
	pdfkeywords={positive scalar curvature, nonnegative Ricci curvature, heat kernel, Nash entropy, volume growth}
}
\theoremstyle{plain}
\newtheorem{theorem}{Theorem}[section]

\newaliascnt{conjecture}{theorem}

\aliascntresetthe{conjecture}

\newaliascnt{corollary}{theorem}
\newtheorem{corollary}[corollary]{Corollary}
\aliascntresetthe{corollary}

\newaliascnt{lemma}{theorem}
\newtheorem{lemma}[lemma]{Lemma}
\aliascntresetthe{lemma}

\newaliascnt{fact}{theorem}

\aliascntresetthe{fact}

\newaliascnt{claim}{theorem}

\aliascntresetthe{claim}

\newaliascnt{proposition}{theorem}
\newtheorem{proposition}[proposition]{Proposition}
\aliascntresetthe{proposition}

\theoremstyle{definition}

\newaliascnt{definition}{theorem}
\newtheorem{definition}[definition]{Definition}
\aliascntresetthe{definition}

\newaliascnt{example}{theorem}
\newtheorem{example}[example]{Example}
\aliascntresetthe{example}

\theoremstyle{remark}

\newaliascnt{remark}{theorem}
\newtheorem{remark}[remark]{Remark}
\aliascntresetthe{remark}

\numberwithin{equation}{section}

\newcommand{\Ric}{\operatorname{Ric}}
\newcommand{\Sc}{\operatorname{Scal}}
\newcommand{\Vol}{\operatorname{Vol}}
\newcommand{\tr}{\operatorname{tr}}
\newcommand{\diver}{\operatorname{div}}
\newcommand{\hes}{\operatorname{Hess}}
\newcommand{\dd}{\,d}

\newcommand{\Wpole}{W^{\mathrm{pole}}}
\newcommand{\supp}{\operatorname{supp}}
\newcommand{\End}{\operatorname{End}}

\begin{document}
\title[Gromov's volume conjecture]{Heat kernel geometry and Gromov's volume growth conjecture}
\author[JG]{Jian Ge}
\address[]{Beijing Normal University, Beijing, China}
\email{jge@bnu.edu.cn}
\begin{abstract}
	In 1986, Gromov asked whether every complete noncompact $n$-dimensional Riemannian manifold with nonnegative Ricci curvature and scalar curvature at least one satisfies:
	\[
		\Vol_g	(B(p, R))\le C_{n}R^{n-2}
	\]
	for all $p\in M$ and $R>0$. We answer this question affirmatively using the heat-kernel Fisher metric and Nash entropy.
\end{abstract}
\maketitle

\section{Introduction}
Let $(M^{n}, g)$ be a connected, complete, noncompact $n$-dimensional Riemannian manifold. Under nonnegative sectional curvature and the scalar-curvature bound $\Sc\ge 1$, Gromov stated the uniform estimate
\begin{equation*}
	\sup_{p} \Vol(B(p, R)) \le C_{n} R^{n-2},
	\qquad R>1,
\end{equation*}
and asked whether it extends to nonnegative Ricci curvature
\cite[p.~114, (5') and \S2.A(b)]{Gro1986}. A frequently used later formulation of Gromov's question asks whether, for a fixed base point $p$,
\begin{equation*}
	\limsup_{R\to \infty}\frac{\Vol(B(p, R))}{ R^{n-2}}<\infty,\qquad \Ric_{g}\ge 0, \Sc_{g}\ge 1.
\end{equation*}
In dimension $n=3$, Munteanu and Wang proved the stronger uniform form \cite{MW2026}: there exists a universal constant $C$ such that, for every $p\in M$ and every $R>0$,
\begin{equation*}
	\Vol(B(p, R)) \le C R.
\end{equation*}
In higher dimensions, after normalizing the scalar-curvature lower bound to $n(n-1)$, Wang--Xie--Zhu--Zhu proved the same exponent after taking the infimum over centers at each scale,
\begin{equation*}
	\inf_{p\in M}\Vol(B(p,R))\le c(n)R^{n-2}, \qquad R\ge0,
\end{equation*}
see \cite[Theorem~1.9]{WXZZ2024}, the minimizing center may depend on $R$.

In this note, we answer this question affirmatively in the uniform form.
\begin{theorem}\label{thm:Main}
	Let $(M^{n}, g)$, $n\ge3$, be a complete, connected, noncompact Riemannian manifold. If
	\begin{equation*}
		\Ric_{g}\ge 0, \quad \Sc\ge 1,
	\end{equation*}
	then there exists a constant $C_{n}<\infty$ depending only on the dimension $n$, such that
	\begin{equation}\label{eq:Main01}
		\Vol_{g}(B(p, R))\le C_{n} R^{n-2},
	\end{equation}
	for every $p\in M$ and every $R>0$.
\end{theorem}

The exponent $n-2$ in \eqref{eq:Main01} is sharp. In fact, for $n\ge 3$, let $M=S^{2}_{a}\times \mathbb{R}^{n-2}$ have the product metric, where $S_{a}^{2}$ denotes the round $2$-sphere of radius $a$ in $\mathbb{R}^{3}$. Let $\omega_{n-2}$ denote the Euclidean volume of the unit ball in $\mathbb{R}^{n-2}$. Then $\Ric\ge 0$, $\Sc=2/a^{2}$, and
\begin{equation*}
	\Vol(B(p, R))\sim 4\pi a^{2} \omega_{n-2} R^{n-2}, \quad R\to \infty.
\end{equation*}
We refer to this example as the \emph{sharp model}.

We sketch the main geometric idea behind the proof. Our guiding geometric picture is to interpret Gromov's conjectural exponent as a loss of two macroscopic dimensions. More precisely, the conjecture suggests that a complete manifold with $\Ric\ge 0$ and $\Sc\ge 1$ should behave, at large scales, as if its effective dimension were at most $n-2$. We find that this dimension loss can be captured by the heat flow. We use the pole-variable heat covariance and its trace defect:
\begin{equation*}
	\mathsf{G}_{x}(t)
	=2t\int_{M}d_{x}h_{x}\otimes d_{x}h_{x}\dd\mu_{x,t},
	\qquad \frac{2}{t}\mathsf{E}_{x}(t)=n-\tr_{g}\mathsf{G}_{x}(t).
\end{equation*}
More precisely, if $\Phi_{t}(x)=H(x,\cdot,t)\dd\Vol$ is the heat kernel map into probability densities and $g_{F}$ is the Fisher information metric, then 
\begin{equation*}
	(\Phi^{*}_{t}g_{F})_{x}(v, w)=
	\int_{M}d_{x}\log H(v) d_{x}\log H(w) H(x,y,t)\dd\Vol_{y},
\end{equation*}
and $\mathsf{G}=2t\Phi^{*}_{t}g_{F}$. This heat-kernel Fisher metric was studied in \cite{ISS2008}.

The trace of $\mathsf{G}_{x}(t)$ measures the number of directions visible to the heat flow at scale $\sqrt{t}$, while $2\mathsf{E}_{x}(t)/t$ measures the missing directions. This quantity vanishes on Euclidean space and tends to $2$ on the sharp model $S_{a}^{2}\times\mathbb{R}^{n-2}$. Under $\Ric\ge 0$, \autoref{prop:Gbounds} gives $0\le\mathsf{G}\le g$, so this defect is nonnegative. The proof is then organized around the propagation of $\mathsf{E}$ through two source equations:
\begin{equation*}
	\square_{x}\mathsf{E}=\frac{1}{2}\mathsf{Q},
	\qquad \square_{x}(-S)=\frac{\mathsf{E}}{t^{2}}.
\end{equation*}

The first identity is the centered form of the raw pole-variable production calculation proved in \autoref{prop:PolewiseDefectSource}.

The curvature assumptions give the following heat-averaged lower bound for the centered source 
\begin{equation*}
	P_{s}\bigl(\mathsf{Q}_{\bullet}(t)\bigr)(z)
	\ge 2-C_{n}\left(t^{-1}+s^{-1}\right)^{1/3},
\end{equation*}
as proved in \autoref{thm:two-unit}. The mechanism behind this estimate is that a region where $\mathsf{Q}<2-\varepsilon$ carries a distinguished lowest covariance eigenline with a definite amount of Ricci curvature; the weighted Weitzenb\"{o}ck estimate shows that such a region has small heat-averaged mass. Applying the minimal heat-potential comparison first to $\mathsf{E}$ and then to $-S$ propagates these two units through the two source equations and yields
\begin{equation*}
	S_{x}(t)\le-\log t+C_{n}.
\end{equation*}
Finally, the Li--Yau heat-kernel estimate converts this entropy decay into
\begin{equation*}
	\Vol(B(x,R))\le C_{n}R^{n-2},
\end{equation*}
which proves \autoref{thm:Main}.

Although the two pole-variable related expressions appeared at the formal level in some papers, they do not by themselves justify the present argument. Ni discusses the complete noncompact fixed-pole formulas in the section entitled ``Extensions and the value of $\mu(0)$'' \cite{Ni2004b}, while Bamler's relevant basepoint calculation is made for smooth compact Ricci flows settings, cf. \cite[Section~5]{Bam2020}. Under our hypotheses, Kotschwar's pole-gradient estimate, the Li--Yau Harnack and Gaussian bounds, and local parabolic estimates give uniform control of all pole-variable heat-kernel. This is used to prove that $\mathsf{G}$, $\mathsf{E}$, its centered Hessian variance $D$, and $\mathsf{Q}$ are smooth and finite on positive time. On a spectral-gap region the lowest-eigenline projection is then smooth by ordinary Riesz calculus. A smooth invariant matrix cutoff keeps the geometric argument entirely at the classical level. After this regularity step, the remaining noncompact issues are the entropy tails and heat-potential comparisons.

The paper is organized as follows. Section~2 recalls the Fisher interpretation of the heat covariance $\mathsf{G}$, introduces the trace defect $\mathsf{E}$ and the centered source $\mathsf{Q}$, and records their pointwise production equation. Section~3 contains the geometric core: a smooth invariant spectral cutoff isolates the lowest eigenline, and a weighted Weitzenb\"{o}ck estimate yields the two-unit estimate \autoref{thm:two-unit}. Section~4 recalls the Nash entropy and its pole-variable equation. In Section~5, the minimal heat-potential comparison propagates the source estimate first to $\mathsf{E}$ and then to $-S$, completing the proof together with the Li--Yau estimate. Appendix~A distinguishes pole-variable identities from fixed-pole output identities. Appendix B proves the uniform estimates of derivatives of heat kernel and the positive-time smoothness of $\mathsf{G}$, $\mathsf{E}$, $D$, and $\mathsf{Q}$.

\section{The heat flow quantities}
\subsection{Notations and conventions}
We always assume $(M^{n}, g)$ is a connected, smooth, complete, noncompact Riemannian manifold with $\Ric \ge 0$ and $\Sc\ge 1$. The dimension $n\ge 3$. We now fix the conventions. Our curvature convention is
\begin{equation*}
	R(X, Y)Z =\nabla_{X}\nabla_{Y}Z - \nabla_{Y}\nabla_{X}Z- \nabla_{[X, Y]}Z,
\end{equation*}
\begin{equation*}
	\Ric(X, X) = \sum_{i} \left\langle R(e_{i}, X) X, e_{i} \right\rangle.
\end{equation*}
If $A$ is a covariant symmetric two-tensor, then $A^{\sharp}$ is the self-adjoint endomorphism obtained by raising its first index using the metric. For two such tensors, we define:
\begin{equation*}
	\left\langle A, B \right\rangle = \tr (A^{\sharp}B^{\sharp}),\qquad |A|^{2}= \left\langle A, A \right\rangle.
\end{equation*}
We use the same letter for a symmetric tensor and its raised endomorphism when the type is clear. For example, if $\Pi$ is an orthogonal projection onto a linear subspace, then we write $\left\langle \Ric, \Pi \right\rangle = \tr (\Ric^{\sharp} \Pi)$.

For a linear map $T:V\to W$, where $V$ is a finite-dimensional inner-product space and $W$ is a Hilbert space, we distinguish the operator norm and the Hilbert--Schmidt norm:
\begin{equation*}
	\|T\|_{\mathrm{op}}
	:=\sup_{|v|=1}|Tv|,
	\qquad
	\|T\|_{\mathrm{HS}}^{2}
	:=\sum_{i=1}^{\dim V}|Te_{i}|^{2},
\end{equation*}
where $\{e_{i}\}$ is any orthonormal basis of $V$. For a covariant symmetric two-tensor $A$, these norms are understood through the associated self-adjoint endomorphism $A^{\sharp}$:
\begin{equation*}
	\|A\|_{\mathrm{op}}
	:=\|A^{\sharp}\|_{\mathrm{op}},
	\qquad
	|A|:=\|A^{\sharp}\|_{\mathrm{HS}}.
\end{equation*}
Thus, if $\lambda_{1},\ldots,\lambda_{n}$ are the eigenvalues of $A^{\sharp}$, then
\begin{equation*}
	\|A\|_{\mathrm{op}}=\max_{i}|\lambda_{i}|,
	\qquad
	|A|^{2}=\sum_{i=1}^{n}\lambda_{i}^{2}.
\end{equation*}
In particular, $|g|^{2}=n$ and $\|g\|_{\mathrm{op}}=1$. Unless an operator norm is explicitly indicated, all tensor norms below are Hilbert--Schmidt norms.

For a Borel set $U$, $\mathbf{1}_{U}$ is the indicator function, and $\supp(f)$ is the closed support of $f$.

The Laplacian is $\Delta = \diver \nabla$; hence $-\Delta$ is nonnegative on $L^{2}(M)$. The forward heat operator (the subscripts specify the spatial variable when needed) and its formal adjoint are denoted by
\begin{equation*}
	\square := \partial_{t}-\Delta,\quad \square_{x} := \partial_{t} - \Delta_{x},\quad \square^{*}_{x}:=-\partial_{t} - \Delta_{x}.
\end{equation*}
The minimal heat kernel of $M$ is denoted by 
\begin{equation*}
	H(x,y,t)=H_{x}(y, t)=e^{-h(x,y,t)}=e^{-h_{x}(y, t)}.
\end{equation*}
Under $\Ric\ge0$, the minimal heat kernel is unique, strictly positive, smooth, and symmetric. Therefore
\begin{equation*}
	H(x,y,t)=H(y,x,t),\quad \square_{x}H=\square_{y}H=0,
\end{equation*}
and $M$ is stochastically complete:
\begin{equation}\label{eq:Conservation}
	\int_{M}H_{x}(y,t)\dd\Vol_{y}=1.
\end{equation}
Therefore $\dd\mu_{x,t}(y)=H_{x}(y,t)\dd\Vol_{y}$ is a probability measure on $M$. The heat semigroup $P_{t}=e^{t\Delta}$ acts on a function via
\begin{equation*}
	P_{t}f(x) = \int_{M} H(x, y, t) f(y)\dd\Vol_{y}=\int_{M}f(y)\dd\mu_{x,t}(y).
\end{equation*}
In an expression like $P_{s}(F_{\bullet}(t))(z)$, the bullet denotes the spatial variable on which $P_{s}$ acts:
\begin{equation*}
	P_{s}(F_{\bullet}(t))(z) = \int_{M} H(z, x, s)F_{x}(t) \dd\Vol_{x}.
\end{equation*}
We call $x$ the \emph{pole variable} and $y$ the \emph{output variable}. Although the heat kernel is symmetric, we distinguish these roles explicitly. The notation $d_{x}h_{x}$ and $\hes_{x}h_{x}$ means that the derivatives are taken in the pole variable before $x$ is frozen; explicitly,
\begin{equation*}
	\hes_{x}h_{x}:=\nabla_{x}(d_{x}h_{x})
\end{equation*}
is a covariant symmetric two-tensor at $x$. A distributional identity $\square u =f$ means that
\begin{equation*}
	\int u \square^{*}\xi = \int f\xi
\end{equation*}
for every compactly supported smooth spacetime test function $\xi$.

\subsection{Heat covariance and trace defect}

For the heat kernel map $\Phi_{t}(x)=H(x,\cdot,t)\dd\Vol$, the Fisher information metric is
\begin{equation*}
	(\Phi_{t}^{*}g_{F})_{x}=
	\int_{M}d_{x}\log H\otimes d_{x}\log H\dd\mu_{x,t}.
\end{equation*}
Thus the tensor below is exactly $2t\Phi_{t}^*g_F$, the normalized Fisher pullback. Heat-kernel Fisher metrics of this form were studied in \cite{ISS2008}.

\begin{definition}\label{def:HeatCovariance}
	For fixed $x\in M$ and $t>0$, we set the \emph{heat covariance} or \emph{heat metric}:	
	\begin{equation}\label{eq:Covariance}
		\mathsf{G}_{x}(t):=2t\int_{M} d_{x}h \otimes d_{x}h \dd\mu_{x,t}.
	\end{equation}
\end{definition}
	The factor $2t$ is intrinsic to the normalization: on Euclidean space $\mathbb{R}^{n}$,
\[
	d_{x}h=\frac{x-y}{2t},
\]
and the Gaussian covariance gives $\mathsf{G}_{x}(t)=g_{x}$. Therefore $g_{x}-\mathsf{G}_{x}$ measures the loss, relative to the Euclidean space, of the normalized covariance at time $t$ at $x$.

\autoref{prop:Gbounds} proves that the integral in the definition of $\mathsf{G}$ is finite and bounded
\begin{equation}\label{eq:2:G_{t}emp_Bounds}
	0\le \mathsf{G}_{x}(t) \le g_{x};
\end{equation}
therefore all tensors defined from $\mathsf{G}$ below are ordinary finite tensors. We also set
\begin{equation}\label{eq:2:Excess}
	\mathsf{E}_{x}(t)=\frac{t}{2}(n-\tr_{g} \mathsf{G}_{x}(t)),
\end{equation}
which we call the pointwise heat-dimension defect. Its relation to earlier entropy formulas is summarized in Appendix A.

By \eqref{eq:2:G_{t}emp_Bounds},
\begin{equation}\label{eq:2:E_Bound}
	0\le \mathsf{E}_{x}(t)\le \frac{nt}{2}.
\end{equation}

\begin{example}
	To illustrate the defect $\mathsf{E}$, we calculate it for the sharp model $M=S_{a}^{2}\times \mathbb{R}^{n-2}$. Write $x=(p, z)$ and $y=(q, w)$. The product heat kernel is 
	\begin{equation*}
		H_{M}(x, y, t) = H_{S_{a}^{2}}(p, q, t)\ (4\pi t)^{\frac{-(n-2)}{2}}\exp \left( -\frac{|z-w|^{2}}{4t} \right).
	\end{equation*}
	Put $h_{S_{a}^{2}}=-\log H_{S_{a}^{2}}$. The mixed covariance term vanishes, and rotational symmetry of $S^{2}_{a}$ gives
	\begin{equation*}
		\mathsf{G}_{(p,z)}(t) = \lambda_{a}(t) g_{S_{a}^{2}} \oplus g_{\mathbb{R}^{n-2}},
	\end{equation*}
	where
	\[
		\lambda_{a}(t) = t\int_{S_{a}^{2}} H_{S_{a}^{2}}(p, q,t) |d_{p}h_{S_{a}^{2}}|^{2} \dd \Vol_{S_{a}^{2}}(q).
	\]
	Therefore $\lambda_{a}(t)\to 1$ as $t\downarrow0$. If $\nu_{1}$ is the first positive spherical eigenvalue, compact spectral convergence gives $H_{S^{2}_{a}}=\Vol(S^{2}_{a})^{-1}+O_{C^{2}}(e^{-\nu_{1}t})$ as $t\to\infty$, and hence $\lambda_{a}(t)=O(te^{-2\nu_{1}t})$ and $t\lambda_{a}'(t)=o(1)$. It follows that 
	\begin{equation*}
		\mathsf{G}^{\sharp}_{(p,z)}(t) \to 0_{T_{p}S_{a}^{2}} \oplus I_{\mathbb{R}^{n-2}}, \qquad t\to \infty.
	\end{equation*}
	Thus the large-time covariance retains precisely the $n-2$ Euclidean directions and loses the two spherical directions. In particular,
	\begin{equation*}
		\tr \mathsf{G}(t) \to n-2,\quad \Delta_{x} \mathsf{E}=0,\quad \square_{x}\mathsf{E}(t)\to 1.
	\end{equation*}
\end{example}

\subsection{Covariance bounds}
\begin{lemma}[\cite{BGL2014}]\label{lem:reverse}
	For every $t>0$, if $f\in C_{c}^{\infty}(M)$, then 
	\begin{equation}\label{eq:reversePoin1}
		P_{t}(f^{2}) - (P_{t}f)^{2} \ge 2t |\nabla P_{t}f|^{2}.
	\end{equation}
	If $f=c+\varphi$ is positive on $M$, where $c>0$ is a constant and $\varphi\in C_{c}^{\infty}(M)$, then
	\begin{equation}\label{eq:reverseLog}
		P_{t}(f\log f) - P_{t}f \log(P_{t}f) \ge t \frac{|\nabla P_{t}f|^{2}}{P_{t}f}.
	\end{equation}
\end{lemma}
\begin{proof}
	The condition $\Ric\ge 0$ is precisely $\mathrm{CD}(0,\infty)$. The reverse local Poincar\'{e} and logarithmic Sobolev inequalities under this condition are Theorems 4.7.2(iv) and 5.5.2(v) of \cite{BGL2014}. At curvature parameter $\rho=0$, their coefficients are $2t$ and $t$, respectively, giving \eqref{eq:reversePoin1} and \eqref{eq:reverseLog}.
\end{proof}

\begin{proposition}\label{prop:Gbounds}
	For every $x\in M$ and $t>0$, we have 
	\begin{equation}\label{eq:Gbounds01}
		0\le \mathsf{G}_{x}(t)\le g_{x},\quad \int_{M} d_{x}h_{x} \dd\mu_{x,t}(y) =0.
	\end{equation}
	For every unit vector $v\in T_{x}M$, and every $a\in \mathbb{R}$, $\ell_{v}(y) = d_{x}\log H (v)$. Then 
	\begin{equation}\label{eq:Gbounds02}
		\int_{M}e^{a\ell_{v}(y)} \dd\mu_{x,t}(y)\le e^{a^{2}/(4t)}.
	\end{equation}
	therefore,
	\begin{equation}\label{eq:Gbounds03}
		\int_{M}|d_{x}h_{x}|^{2} \dd\mu_{x,t}(y) \le \frac{n}{2t},\qquad \int_{M}|d_{x}h_{x}|^{4} \dd\mu_{x,t}(y) \le \frac{4n^{2}}{t^{2}}.
	\end{equation}
\end{proposition}
\begin{proof}
	Fix $x\in M$ and $t>0$, and write $\mu=\mu_{x,t}$ and
	$\mu(f)=\int_{M} f\dd\mu$.

	\smallskip
	\noindent\emph{Step 1: Fisher tensor bound.}
	For $v\in T_{x}M$, set $\ell_v(y)=d_{x}\log H(x,y,t)(v)$. If
	$f\in C_c^\infty(M)$, then
	\begin{equation*}
		d(P_{t}f)_{x}(v)=\int_{M} f\ell_v\dd\mu.
	\end{equation*}
	By \eqref{eq:reversePoin1},
	\begin{equation*}
		|d(P_{t}f)_{x}(v)|^{2}\le \frac{|v|^{2}}{2t}
		\int_{M}|f-\mu(f)|^{2}\dd\mu.
	\end{equation*}
	The subspace
	\begin{equation*}
		\mathcal D=\{f-\mu(f):f\in C_c^\infty(M)\}
	\end{equation*}
	is dense in $L_{0}^{2}(\mu)=\{u\in L^{2}(\mu):\mu(u)=0\}$. Hence $L_v(f-\mu(f))=d(P_{t}f)_{x}(v)$ extends to a bounded functional on $L_{0}^{2}(\mu)$ with norm at most $|v|/\sqrt{2t}$. Let $r_v$ be its Riesz representative. For every $f\in C_c^\infty(M)$,
	\begin{equation*}
		\int_{M}fr_v\dd\mu=d(P_{t}f)_{x}(v)=\int_{M}f\ell_v\dd\mu.
	\end{equation*}
	Since $H>0$, $r_v=\ell_v$ almost everywhere. Consequently,
	\begin{equation}\label{eq:score-L2}
		\mu(\ell_v)=0,\qquad
		\int_{M}\ell_v^{2}\dd\mu\le\frac{|v|^{2}}{2t}.
	\end{equation}
	Thus $\mathsf{G}_{x}(v,v)=2t\int_{M}\ell_v^{2}\dd\mu\le|v|^{2}$, which
	proves \eqref{eq:Gbounds01}.

	\smallskip
	\noindent\emph{Step 2: Exponential score bound.}
	Let $|v|=1$ and abbreviate $\ell_v$ by $\ell$. For a bounded
	nonnegative function $F$, set
	\begin{equation*}
		\mathrm{Ent}_\mu(F)=\mu(F\log F)-\mu(F)\log\mu(F),
	\end{equation*}
	with $0\log0=0$. If $F\ge0$ is smooth and bounded and $\mu(F)>0$, choose $0\le\chi_{R}\le1$ in $C_c^\infty(M)$ with $\chi_R\to1$, and apply \eqref{eq:reverseLog} to $\varepsilon+\chi_{R} F$. Letting $R\to\infty$ and then $\varepsilon\downarrow0$ gives
	\begin{equation}\label{eq:score-reverse-entropy}
		\mathrm{Ent}_\mu(F)\ge \frac{t}{\mu(F)}\mu(F\ell)^{2}.
	\end{equation}
	Indeed, \eqref{eq:score-L2} and Cauchy--Schwarz give $\mu(\chi_{R} F\ell)\to\mu(F\ell)$, while the remaining terms converge by dominated convergence.

	For $a\in\mathbb{R}$ and $N\in\mathbb{N}$, define the smooth bounded truncation
	\begin{equation*}
		u_N=\frac{e^{a\ell}}{1+e^{a\ell-N}}.
	\end{equation*}
	Then $0<u_{N} \le e^{N}$, $\log u_{N}\le a\ell$, and $u_N\uparrow e^{a\ell}$. Moreover, $u_N\ell\in L^1(\mu)$ by \eqref{eq:score-L2}, and $u_{N}\log u_{N} \in L^1(\mu)$. Applying \eqref{eq:score-reverse-entropy} to $u_{N}$ and putting $q_{N}=\mu(u_{N}\ell)/\mu(u_{N})$ yields
	\begin{equation*}
		\log\mu(u_{N}) \le aq_{N}-t q_{N}^{2}\le\frac{a^{2}}{4t}.
	\end{equation*}
	Monotone convergence now proves \eqref{eq:Gbounds02}.

	\smallskip
	\noindent\emph{Step 3: Moment bounds.}
	Taking the trace of $\mathsf{G}_{x}(t)\le g_{x}$ gives the second-moment bound in \eqref{eq:Gbounds03}. For $\lambda>0$, Markov's inequality and \eqref{eq:Gbounds02}, optimized at $a=2t\lambda$, give 
	\begin{equation*}
		\mu\{\ell\ge\lambda\}\le e^{-t\lambda^{2}}.
	\end{equation*}
	Applying this to $-\ell$ and using the layer-cake formula gives
	\begin{equation*}
		\int_{M}|\ell|^4\dd\mu \le8\int_{0}^\infty\lambda^3e^{-t\lambda^{2}}\dd\lambda =\frac4{t^{2}}.
	\end{equation*}
	For an orthonormal basis $\{e_{i}\}_{i=1}^n$ of $T_{x}M$,
	\begin{equation*}
		|d_{x}h_{x}|^4\le n\sum_{i=1}^n |d_{x}\log H(x,y,t)(e_{i})|^4.
	\end{equation*}
	Integrating proves the fourth-moment bound in \eqref{eq:Gbounds03}.
\end{proof}

The quadratic estimate in \autoref{prop:Gbounds}, and hence $\mathsf{G}\le g$, is a classical consequence of the reverse Poincar\'{e} inequality. In the compact super Ricci-flow setting, the same bound is stated directly in \cite[Proposition~4.2]{Bam2020}. The exponential and fourth-moment bounds are retained because the latter enters the quantitative pole-gradient estimate \eqref{eq:Grad_G}.

\subsection{Local identity behind the defect source}

We next write the production density of $\mathsf{E}$ in a centered form adapted to the quantitative estimate in Section~3.

The raw local production identity is a rearrangement of Ni's pointwise $W$-density formula; see \cite[Lemma~2.2]{Ni2004b}. We record the calculation for fixed $y$ because the centered tensor variance decomposition that follows determines the expression for $\mathsf{Q}$ used later.

Since $\mathsf{E}$ is defined through $y$-integration, we begin with a fixed $y$ Bochner formula. The following calculation is the main reason for the definition of the parabolic production density of $\mathsf{E}$.
\begin{lemma}\label{lem:local-polewise}
	Fix $y\in M$, let $u=u(x,t)>0$ solve $\square_{x}u=0$ on a positive time cylinder domain, and put $h=-\log u$. All derivatives below are taken with respect to $x$. Then 
	\begin{equation}\label{eq:Bochner}
	\begin{aligned}
		\square & \left( t^{2}u\left( \frac{n}{2t} - |\nabla h|^{2} \right) \right) &\\
		&= 2t^{2}u \left( \left| \hes h - \frac{1}{2t}g \right|^{2} + \Ric (\nabla h, \nabla h) \right) + 2 t \diver (u \nabla h).
	\end{aligned}
	\end{equation}
\end{lemma}
\begin{proof}
	The proof is a direct calculation. $h$ satisfies the logarithmic heat equation 
	\begin{equation*}
		\square h = - |\nabla h|^{2}.
	\end{equation*}
	By Bochner identity and the product rule, we have
	\begin{equation*}
		\square |\nabla h|^{2} = - 2\left\langle \nabla h, \nabla |\nabla h|^{2} \right\rangle - 2 |\hes h|^{2}
		-2 \Ric(\nabla h, \nabla h),
	\end{equation*}
	\begin{equation*}
		\square(uf) = u \square f + f \square u - 2 \left\langle \nabla u, \nabla f \right\rangle.
	\end{equation*}
	Taking $f=|\nabla h|^{2}$ and using $\square u=0$ and $\nabla u=-u\nabla h$, we see that the first-order terms cancel, and
	\begin{equation*}
		\square(u|\nabla h|^{2}) = -2 u \left( |\hes h|^{2} + \Ric (\nabla h, \nabla h) \right).
	\end{equation*}
	It follows that
	\begin{equation*}
		\begin{aligned}
		 &\square \left( t^{2}u\left( \frac{n}{2t} - |\nabla h|^{2} \right) \right)\\
		 &= \frac{n}{2}u -2tu |\nabla h|^{2}
		 + 2t^{2}u \left( |\hes h|^{2} + \Ric (\nabla h, \nabla h) \right)\\
		 &= 2t^{2}u\left( \left| \hes h - \frac{1}{2t}g \right|^{2} 
		 + \Ric(\nabla h, \nabla h)\right) +2tu (\Delta h - |\nabla h|^{2}).
		\end{aligned}
	\end{equation*}
	Combining with
	\begin{equation*}
		u(\Delta h - |\nabla h|^{2}) 
		= u \Delta h + \left\langle \nabla u, \nabla h \right\rangle = \diver (u\nabla h),
	\end{equation*}
	we get \eqref{eq:Bochner}.
\end{proof}

Motivated by the preceding calcuation, we write its density in the following centered form:
\begin{definition}\label{def:D}
	For $x\in M$ and $t>0$, define the centered Hessian variance
	\begin{equation}\label{eq:D}
		D_{x}(t):=\int_{M} H(x, y, t)
		\left|\hes_{x}h_{x}(y,t)-\frac{\mathsf{G}_{x}(t)}{2t}\right|^{2} \dd\Vol_{y}.
	\end{equation}
\end{definition}

\begin{definition}\label{def:Q}
	For $x\in M$ and $t>0$, we define the \emph{defect source} or \emph{defect density}:
	\begin{equation}\label{eq:Q}
		\mathsf{Q}_{x}(t)=4t^{2}D_{x}(t) +|g_{x}-\mathsf{G}_{x}(t)|^{2}
		+2t\left\langle \Ric_{x}, \mathsf{G}_{x}(t) \right\rangle.
	\end{equation}
\end{definition}

The point is to interpret $\mathsf{E}$ as a heat-dimension defect and to center its production source in the form \eqref{eq:Q}. The key estimates are the heat-averaged two-unit lower bound for $\mathsf{Q}$ and its propagation to entropy decay. It is the crucial place that the two curvature conditions entering our estimates. 

\begin{proposition}[Smoothness of the heat quantities]\label{prop:HeatTensorRegularity}
	The tensor field $\mathsf{G}$ and the scalar fields $\mathsf{E}$, $D$, and $\mathsf{Q}$ are finite and smooth on $M\times(0,\infty)$. On every $K\Subset M$ and $0<a<b<\infty$, the output-variable integrands that define $\mathsf{G}$ and $D$, together with all their pole-space-time	derivatives, have an integrable majorant independent of $(x,t)\in K\times[a,b]$. Consequently, the pole-space-time	differentiations used below may be passed through the output integrals.
\end{proposition}
\begin{proof}
	The proof is given in \autoref{proof:HeatTensorRegularity}.
\end{proof}

\subsection{Pointwise defect source}

Once the output-variable tails of all pole derivatives are controlled, the preceding calculation is classical rather than merely distributional.

\begin{proposition}\label{prop:PolewiseDefectSource}
	For every $x\in M$ and $t>0$,
	\begin{equation}\label{eq:mean-score-hessian}
		\int_{M}H_{x}d_{x}h_{x}\dd\Vol_{y}=0,
		\qquad
		\int_{M}H_{x}\hes_{x}h_{x}\dd\Vol_{y}=\frac{\mathsf{G}_{x}(t)}{2t}.
	\end{equation}
	Moreover,
	\begin{equation}\label{eq:Box_E=Q}
		\square_{x}\mathsf{E}_{x}(t)=\frac{1}{2}\mathsf{Q}_{x}(t)
	\end{equation}
	pointwise, and hence also in the sense of distributions.
\end{proposition}
\begin{proof}
	By \autoref{prop:HeatTensorRegularity}, all differentiations and 	output-variable integrations below are classical. We apply 	\autoref{lem:local-polewise} to $u(x,t)=H(x,y,t)$.

	First, by \eqref{eq:Conservation} and \eqref{eq:Covariance} we have
	\begin{equation*}
	\int_{M} t^{2} \left( \frac{n}{2t} - |d_{x}h_{x}|^{2} \right) H_{x}\dd\Vol_{y} 
	= t^{2}\left( \frac{n}{2t} - \frac{\tr \mathsf{G}_{x}(t)}{2t} \right) = \mathsf{E}_{x}(t).
	\end{equation*}

	Integrating \eqref{eq:Bochner} in the $y$-variable gives
	\begin{equation}\label{eq:Bochner02}
	\begin{aligned}
		\square_{x}\mathsf{E}_{x}(t) = & 2t^{2}\int_{M} \left( \left| \hes_{x}h_{x} - \frac{1}{2t}g_{x}\right|^{2} 
		+ \Ric_{x}(\nabla_{x}h_{x}, \nabla_{x}h_{x}) \right)\dd\mu_{x,t}\\
		&+ 2t \diver_{x} \int_{M} \nabla_{x}h_{x}\dd\mu_{x,t}.
	\end{aligned}
	\end{equation}
	Note that the divergence theorem cannot be applied to the last term, since the integration variable is $y$. Nevertheless, we have
	\begin{equation*}
	\int_{M}\nabla_{x}h_{x}\dd\mu_{x,t} = -\int_{M}\nabla_{x}H_{x} \dd\Vol_{y} = -\nabla_{x}\int_{M}H_{x}\dd\Vol_{y}=0.
	\end{equation*}
	This proves the first identity in \eqref{eq:mean-score-hessian}. Differentiating \eqref{eq:Conservation} twice gives
	\begin{equation*}
	0=\hes_{x}\int_{M}H_{x}\dd\Vol_{y} = \int_{M} H_{x}(d_{x}h_{x}\otimes d_{x}h_{x} - \hes_{x}h_{x}) \dd\Vol_{y},
	\end{equation*}
	and therefore
	\begin{equation}\label{eq:Bochner03}
	\int_{M} \hes_{x}h_{x} \dd\mu_{x,t} = \int_{M} d_{x}h_{x}\otimes d_{x}h_{x} \dd\mu_{x,t} = \frac{\mathsf{G}_{x}(t)}{2t}.
	\end{equation}
	This is the second identity in \eqref{eq:mean-score-hessian}. Define
	\begin{equation*}
	A_{y}:= \hes_{x}h_{x}(y, t),\quad \bar{A}:= \int_{M} A_{y} \dd\mu_{x,t}(y)=\frac{\mathsf{G}_{x}(t)}{2t}.
	\end{equation*}
	We view $A_{y}$ and $\bar{A}$ as elements in $\mathrm{Sym}^{2} (T^{*}_{x}M)$. Since $\mu_{x,t}$ is a probability measure, the variance identity gives:
	\begin{equation*}
	\begin{aligned}
		\int_{M}\left| A_{y} - \frac{g_{x}}{2t} \right|^{2}\dd \mu_{x,t} 
		&= \int_{M}|A_{y} - \bar{A}|^{2}\dd\mu_{x,t} + |\bar{A} - \frac{g_{x}}{2t}|^{2} \\
	&= \int_{M} \left| A_{y} - \frac{\mathsf{G}_{x}(t)}{2t} \right|^{2}\dd\mu_{x,t} + \frac{1}{4t^{2}} |g_{x} - \mathsf{G}_{x}(t)|^{2}.
	\end{aligned}
	\end{equation*}
	Similarly:
	\begin{equation*}
	\begin{aligned}
	4t^{2}\int_{M} \Ric_{x}(\nabla_{x}h_{x}, \nabla_{x}h_{x})\dd\mu_{x,t} 
	&= 4t^{2} \left\langle\Ric_{x}, \int_{M}d_{x}h_{x}\otimes d_{x}h_{x} \dd\mu_{x,t} \right\rangle\\
	&=2t \left\langle \Ric_{x}, \mathsf{G}_{x}(t) \right\rangle.
	\end{aligned}
	\end{equation*}

Consequently,
	\begin{equation}\label{eq:Bochner04}
	\begin{aligned}
	 2\square_{x}\mathsf{E}_{x}(t) &=4t^{2} \int_{M} \left( \left| \hes_{x}h_{x} - \frac{1}{2t}g_{x}\right|^{2} 
		+ \Ric_{x}(\nabla_{x}h_{x}, \nabla_{x}h_{x}) \right) \dd\mu_{x,t}\\
	 &= 4t^{2}\int_{M} \left( \left| \hes_{x}h_{x} - \frac{\mathsf{G}_{x}(t)}{2t}\right|^{2}\right) \dd\mu_{x,t}\\
	 &+ |g_{x}-\mathsf{G}_{x}(t)|^{2} + 2t \left\langle \Ric_{x}, \mathsf{G}_{x}(t) \right\rangle.
	\end{aligned}
	\end{equation}

	By \eqref{eq:D} and \eqref{eq:Q}, the last identity is precisely \eqref{eq:Box_E=Q}.
\end{proof}

\section{The two-unit defect source}

The following heat-averaged two-unit lower bound for the centered source $\mathsf{Q}$ is the key to our proof. It shows that the assumptions $\Ric\ge0$ and $\Sc\ge1$ force at least two units in the source after heat averaging.

\begin{theorem}[Two-unit defect]\label{thm:two-unit}
	There exists a constant $C_{n}<\infty$ such that for every $t\ge 1$, $s>0$, and $z\in M$,
	\begin{equation}\label{eq:two-unit}
		P_{s}(\mathsf{Q}_{\bullet}(t))(z) \ge 2 - C_{n}\left( t^{-1} + s^{-1} \right) ^{1/3}.
	\end{equation}
\end{theorem}

\begin{remark}
	If $P_{s}(\mathsf{Q}_{\bullet}(t))(z)=+\infty$, then \eqref{eq:two-unit} holds automatically. Thus only the finite weighted-energy case requires an estimate.
\end{remark}

\subsection{The first unit and the spectral gap}

	We now turn to the geometric side of the proof. The scalar-curvature lower bound first forces one unit of the defect source. If $\mathsf{Q}_{x}(t)$ remains below $2-\varepsilon$ for some $0<\varepsilon\le 1/2$ and $t\ge 6/\varepsilon$, then the heat metric has a unique, quantitatively separated lowest eigenline, and most of the Ricci curvature is concentrated in that direction.

We set up the notation needed in this section. For a fixed $t>0$, let
\begin{equation*}
	0\le \lambda_{1}(x, t)\le \cdots\le \lambda_{n}(x, t)\le 1
\end{equation*}
be the ordered eigenvalues of the self-adjoint endomorphism 
\[
	\mathsf{G}_{x}(t)^{\sharp}: T_{x}M\to T_{x}M.
\]
We often write them as $\lambda_{i}$. We set
\begin{equation*}
	\mathcal{U}_{\min}(t):= \left\{ x\in M\ : \ \lambda_{1}<\lambda_{2} \right\}.
\end{equation*}
For $x\in \mathcal{U}_{\min}(t)$, define
\begin{equation*}
	L_{x}(t)=\ker (\mathsf{G}_{x}(t)^{\sharp}-\lambda_{1}(x, t)I),
\end{equation*}
and
\begin{equation*}
	\Pi_{x}(t):=\mathrm{proj}_{L_{x}(t)}\in \mathrm{End}(T_{x}(M)).
\end{equation*}

Thus $L_{x}(t)$ is the eigenspace corresponding to the lowest eigenvalue. It is a smooth line bundle on $\mathcal U_{\min}(t)$, and $\Pi(t)$ is a smooth section of $\End(TM)$ there, as shown in \autoref{prop:SpecProj_Regularity}. If $\eta$ is a local smooth unit section of $L(t)$, then
\begin{equation*}
	\Pi=\eta\otimes \eta^{\flat},\quad
	\left\langle \Ric, \Pi \right\rangle = \Ric(\eta, \eta),
\end{equation*}
and $\Pi$ is independent of the choice of $\eta$. Its covariant derivative is determined by
\begin{equation*}
	(\nabla_{X}\Pi)(V):= \nabla_{X}(\Pi(V))- \Pi(\nabla_{X}(V)),\quad |\nabla \Pi|^{2}:= \sum_{i}\|\nabla_{e_{i}}\Pi\|^{2}_{\mathrm{HS}}.
\end{equation*}

The first unit of $\mathsf{Q}$ is purely algebraic.

\begin{lemma}\label{lem:OneUnit}
	Let $A\ge 0$ and $G$ be two self-adjoint endomorphisms of $\mathbb{R}^{n}$, with $\tr A\ge 1$ and $0\le G\le I$. Then for every $t \ge 1$,
	\begin{equation*}
		|I- G|^{2} + 2t \left\langle A, G \right\rangle \ge 1.
	\end{equation*}
	Applied to $A=\Ric^{\sharp}$ and $G= \mathsf{G}(t)$, this gives
	\begin{equation*}
		0\le 4t^{2}\int_{M} \left|\hes_{x}h_{x}- \frac{\mathsf{G}_{x}(t)}{2t} \right|^{2}\dd \mu_{x,t}(y)\le \mathsf{Q}-1.
	\end{equation*}
\end{lemma}
\begin{proof}
	Let $\lambda_{1}$ be the smallest eigenvalue of $G$. Since $G\ge \lambda_{1}I$, $A\ge 0$ and $\tr A\ge 1$, we have:
	\begin{equation*}
		\left\langle A,G \right\rangle \ge \lambda_{1} \tr A \ge \lambda_{1},
	\end{equation*}
	while $|I-G|^{2}\ge (1-\lambda_{1})^{2}$. Therefore
	\begin{equation*}
		|I-G|^{2}+2t \left\langle A,G \right\rangle \ge (1-\lambda_{1})^{2} + 2t\lambda_{1} \ge 1.
	\end{equation*}
	For $A=\Ric^{\sharp}$, the hypotheses follow from $\Ric\ge 0$, $\Sc\ge 1$, and $0\le \mathsf{G}\le g$.
\end{proof}

\begin{lemma}\label{lem:Sub2}
	Fix $0<\varepsilon \le 1/2$ and suppose $t\ge 6/\varepsilon$ and $\mathsf{Q}\le 2-\varepsilon$. Then
	\begin{equation}\label{eq:sub201}
		\lambda_{1}< \frac{\varepsilon}{6}, \quad 
		\lambda_{2}> \frac{\varepsilon}{3}, \quad 
		\lambda_{2}-\lambda_{1} > \frac{\varepsilon}{6}.
	\end{equation}
	Let $\Pi$ be the rank-one projection onto the $\lambda_{1}$ eigenline. Then
	\begin{equation}\label{eq:sub202}
		\left\langle \Ric, \Pi \right\rangle > \frac{1}{2}.
	\end{equation}
\end{lemma}
\begin{proof}
	Since $\mathsf{G}\ge \lambda_{1}I$, $\Ric \ge 0$, and $\Sc \ge 1$, we have:
	\begin{equation*}
		\lambda_{1} \le \left\langle \Ric, \mathsf{G} \right\rangle \le \frac{2-\varepsilon}{2t}=:a.
	\end{equation*}
	The time assumption gives
	\begin{equation*}
		\lambda_{1} \le a < \frac{\varepsilon}{6}.
	\end{equation*}
	The metric defect term gives:
	\begin{equation*}
		(1-\lambda_{1})^{2} + (1-\lambda_{2})^{2} \le |I -G|^{2} \le 2 - \varepsilon.
	\end{equation*}
	Since $\lambda_{1}\le a$, we have:
	\begin{equation*}
		(1-\lambda_{2})^{2}\le 2-\varepsilon - (1-a)^{2}=1-(a^{2}-2a+\varepsilon).
	\end{equation*}
	Note $r:=a^{2}-2a+\varepsilon$ lies in $(0, 1)$, and $1-\sqrt{1-r}\ge r/2$. Moreover $a<\varepsilon/6$, so
	\begin{equation*}
		r > \varepsilon-\varepsilon/3=2\varepsilon/3.
	\end{equation*}
	Therefore
	\[
		\lambda_2\ge 1-\sqrt{1-r}\ge r/2>\varepsilon/3.
	\]
	Thus $\lambda_{2}-\lambda_{1} > \varepsilon/6$, so $\Pi$ is well defined. Now choose a $\mathsf{G}$-eigenbasis beginning with a unit vector $v$ in the lowest eigenspace. Positivity of Ricci gives
	\begin{equation*}
		\left\langle \Ric, \mathsf{G} \right\rangle \ge \lambda_{2} (\Sc - \Ric (v, v)).
	\end{equation*}
	Since $\left\langle \Ric, \mathsf{G} \right\rangle <1/t$ and $\lambda_{2}> \varepsilon/3$, we have
	\begin{equation*}
		\Sc - \Ric(v, v) < \frac{1}{t\lambda_{2}} < \frac{3}{\varepsilon t} \le \frac{1}{2}.
	\end{equation*}
	Since $\Sc\ge1$, it follows that
	$\Ric(v,v)>\Sc-1/2\ge1/2$. As
	$\Ric(v,v)=\left\langle \Ric,\Pi\right\rangle$, this proves
	\eqref{eq:sub202}.
\end{proof}

\subsection{Covariance and eigenprojection estimates}
Our next goal is to study the sublevel set:
\begin{equation*}
	\Omega_{\varepsilon}(t):=\left\{x\in M:\mathsf{Q}_{x}(t)<2-\varepsilon\right\},
\end{equation*}
which we call the ``bad region.'' By \autoref{lem:Sub2}, every point in $\Omega_{\varepsilon}(t)$ determines a unique lowest covariance eigenline, separated from the remaining eigenspaces by a definite spectral gap, and this eigenline carries a definite amount of Ricci curvature. The bad region need not be empty, and we do not attempt to exclude it pointwise. Instead, we construct a spectral cutoff and show that its mass under a later heat average is quantitatively small. In this sense, the bad region becomes invisible to the heat flow at large scales.

\begin{proposition}\label{prop:SpecProj_Regularity}
	For every $x\in M$ and $t>0$,
	\begin{equation}\label{eq:Grad_G}
		|\nabla \mathsf{G}_{x}(t)|^{2}\le C_{n}(tD_{x}(t)+t^{-1}).
	\end{equation}
	For a fixed $t>0$, if $U$ is an open subset of $M$ on which $\lambda_2-\lambda_{1}>\delta>0$, then the projection $\Pi(t)$ onto the smallest eigenspace is smooth on $U$ and
	\begin{equation}\label{eq:Grad_{p}i}
		|\nabla \Pi|\le \frac{1}{\delta} |\nabla \mathsf{G}|,
	\end{equation}
	pointwise on $U$.
\end{proposition}
\begin{proof}
	Fix $t>0$ and set
	\[
		\mathcal{L}:=L^{2}(M,\dd\Vol_{y}),	\qquad \underline{\mathcal{L}}:=M\times\mathcal{L}.
	\]
	The measure defining $\mathcal L$ is independent of the pole, so $\underline{\mathcal L}$ is a trivial Hilbert bundle. We denote its flat metric connection by $\nabla^{\mathcal L}$.

	For $p\in M$ and $\xi,\eta\in T_{p}M$, define
	\[
		[\mathcal{A}_{t}(p)\eta](y) =-\sqrt{2tH_{p}(y, t)} d_{p}h_{p}(y, t)(\eta),
	\]
	and
	\begin{align*}
		&[\mathcal{B}_{t}(p)(\xi,\eta)](y)\\
		&:=-\sqrt{2tH_{p}(y,t)}
		\left(\hes_{p}h_{p}(y,t)(\xi,\eta)-\frac{1}{2}d_{p}h_{p}(y, t)(\xi)d_{p}h_{p}(y, t)(\eta)\right).
	\end{align*}
	Therefore
	\[
		\mathcal{A}_{t}\in \Gamma(T^{*}M\otimes\underline{\mathcal{L}}),
		\qquad \mathcal{B}_{t}\in\Gamma(T^{*}M\otimes T^{*}M\otimes\underline{\mathcal{L}}).
	\]
	By \autoref{lem:heat-jets}, in every smooth local frame each pole-variable derivative of the components of $\mathcal{A}_{t}$ or $\mathcal{B}_{t}$ is $\sqrt{H}$ times a finite polynomial. Its squared $\mathcal{L}$-norm has a locally uniform integrable majorant in the output variable $y$. Dominated difference quotients therefore show that $\mathcal{A}_{t}$ and $\mathcal{B}_{t}$ are smooth Hilbert-valued tensor fields.

	We now calculate the covariant derivative of $\mathcal{A}_{t}$. Let $U\subset M$ be open, and let $X, Y\in\Gamma^{\infty}(TU)$ be smooth vector fields. The connection induced on $T^{*}M\otimes \underline{\mathcal{L}}$ is
	\[
		(\nabla_{X}\mathcal{A}_{t})(Y) := \nabla_{X}^{\mathcal{L}}\bigl(\mathcal{A}_{t}(Y)\bigr) -\mathcal{A}_{t}(\nabla_{X}^{g}Y),
	\]
	where $\nabla^{g}$ denotes the Levi--Civita connection.

	For every fixed $y\in M$, the identity $h(\cdot, y, t)=-\log H(\cdot,y,t)$ gives
	\[
		X\bigl(\sqrt{H(\cdot,y,t)}\bigr)	=-\frac{1}{2}\sqrt{H(\cdot, y, t)}d_{x}h(X).
	\]
	Moreover,
	\[
		X\bigl(d_{x}h(Y)\bigr)=\hes_{x}h(X,Y)+d_{x}h(\nabla_{X}^{g}Y).
	\]
	Consequently, pointwise on $U$ and in the output variable $y$,
	\begin{align*}
		(\nabla_{X}\mathcal{A}_{t})(Y)= &-\sqrt{2tH}
		\left(\hes_{x}h(X,Y)+d_{x}h(\nabla_{X}^gY)-\frac{1}{2}d_{x}h(X)d_{x}h(Y)\right)\\
		&\quad+\sqrt{2tH}\,d_{x}h(\nabla_{X}^gY)\\
		= &-\sqrt{2tH}\left(\hes_{x}h(X,Y)-\frac{1}{2} d_{x}h(X)d_{x}h(Y)\right)\\
		= &\mathcal{B}_{t}(X,Y).
	\end{align*}
	Hence, with the convention
	\[
		(\nabla\mathcal{A}_{t})(X,Y):=(\nabla_{X}\mathcal{A}_{t})(Y),
	\]
	we have
	\[
		\nabla\mathcal{A}_{t}=\mathcal{B}_{t}.
	\]
	Although vector fields are used in this calculation, the resulting identity is tensorial. In particular, its value at $p$ depends only on $X_{p}$ and $Y_{p}$, and not on the chosen local extensions. For each $p\in M$, consider
	\[
		\mathcal{A}_{t}(p):T_{p}M\longrightarrow\mathcal{L}
	\]
	as a bounded operator. For $\eta,\zeta\in T_{p}M$,
	\begin{align*}
		\left\langle
			\mathcal{A}_{t}(p)\eta, \mathcal{A}_{t}(p)\zeta \right\rangle_{\mathcal{L}}
		&= 2t\int_{M} H_{p}d_{p}h_{p}(\eta)d_{p}h_{p}(\zeta)\dd\Vol_{y}\\
		&= \mathsf{G}_{p}(t)(\eta,\zeta).
	\end{align*}
	Therefore
	\[
		\mathsf{G}_{p}(t)^\sharp =\mathcal{A}_{t}(p)^*\mathcal{A}_{t}(p),
	\]
	where the adjoint is taken with respect to $g_{p}$ and the $\mathcal L$-inner product.

	To differentiate this identity, let $X,Y,Z\in\Gamma^\infty(TU)$. Metric compatibility gives
	\begin{align*}
		(\nabla_{X}\mathsf{G})(Y,Z)=&
		X\left\langle	\mathcal{A}_{t}(Y),\mathcal{A}_{t}(Z)\right\rangle_{\mathcal L}
		-\left\langle\mathcal{A}_{t}(\nabla_{X}^gY),\mathcal{A}_{t}(Z) \right\rangle_{\mathcal L}
		-\left\langle\mathcal{A}_{t}(Y),\mathcal{A}_{t}(\nabla_{X}^gZ) \right\rangle_{\mathcal L}\\
		=&\left\langle(\nabla_{X}\mathcal{A}_{t})(Y),\mathcal{A}_{t}(Z)\right\rangle_{\mathcal L}
		+\left\langle\mathcal{A}_{t}(Y),(\nabla_{X}\mathcal{A}_{t})(Z)\right\rangle_{\mathcal L}\\
		=&\left\langle\mathcal{B}_{t}(X,Y),\mathcal{A}_{t}(Z) \right\rangle_{\mathcal L}
		+\left\langle\mathcal{A}_{t}(Y),\mathcal{B}_{t}(X,Z)\right\rangle_{\mathcal L}.
	\end{align*}

	Now fix $p\in M$ and $\xi\in T_{p}M$, and define the contraction in the first, covariant-derivative slot by
	\[
		\mathcal{B}_{t, p,\xi}:=\mathcal{B}_{t}(p)(\xi, \cdot):T_{p}M\longrightarrow\mathcal L.
	\]
	Since the preceding identity is tensorial, it is equivalently the pointwise operator identity
	\begin{equation}\label{eq:Hilbert-product}
		\bigl(	(\nabla\mathsf{G})_{p}(\xi,\cdot,\cdot)\bigr)^\sharp
		=\mathcal{B}_{t, p,\xi}^*\mathcal{A}_{t}(p)+\mathcal{A}_{t}(p)^*\mathcal{B}_{t, p, \xi}.
	\end{equation}
	Thus every term in \eqref{eq:Hilbert-product} is an endomorphism of $T_{p}M$.

	By \autoref{prop:Gbounds},
	\[
		\|\mathcal{A}_{t}(p)\|_{\mathrm{op}}^{2}=\|\mathcal{A}_{t}(p)^*\mathcal{A}_{t}(p)\|_{\mathrm{op}}
		=\lambda_{\max}\bigl(\mathsf{G}_{p}(t)^\sharp\bigr)\le1.
	\]
	It follows from \eqref{eq:Hilbert-product} that
	\[
		\left\|	\bigl((\nabla\mathsf{G})_{p}(\xi,\cdot,\cdot)\bigr)^\sharp
		\right\|_{\mathrm{HS}}	\le 2\|\mathcal{A}_{t}(p)\|_{\mathrm{op}}\|\mathcal{B}_{t, p,\xi}\|_{\mathrm{HS}}.
	\]

	Inserting $\mathsf{G}_{p}(t)/(2t)$ in the definition of $\mathcal{B}_{t}$ and using $|a+b+c|^{2}\le 3(|a|^{2}+|b|^{2}+|c|^{2})$, we obtain
	\begin{align*}
		|\mathcal{B}_{t}(p)|^{2}& =2t\int_{M}H_{p}
		\left|\hes_{p}h_{p}-\frac{1}{2} d_{p}h_{p}\otimes d_{p}h_{p}\right|^{2}\dd\Vol_{y}\\
		&\le 6tD_{p}(t) +\frac{3}{2t}|\mathsf{G}_{p}(t)|^{2}+\frac{3t}{2}\int_{M} H_{p}|d_{p} h_{p}|^{4} \dd\Vol_{y}\\
		&\le 6tD_{p}(t)+\frac{C_n}{t},
	\end{align*}
	where the last inequality follows from \autoref{prop:Gbounds}.

	Finally, let $\{e_{i}\}_{i=1}^{n}$ be an orthonormal basis of $T_{p}M$. Then
	\begin{align*}
		|(\nabla\mathsf{G}(t))_{p}|^{2}
		&=\sum_{i=1}^{n}\left\|\bigl((\nabla\mathsf{G})_{p}(e_{i},\cdot,\cdot)\bigr)^\sharp\right\|_{\mathrm{HS}}^{2}\\
		&\le 4\|\mathcal{A}_{t}(p)\|_{\mathrm{op}}^{2}\sum_{i=1}^{n}\|\mathcal{B}_{t, p, e_{i}}\|_{\mathrm{HS}}^{2}\\
		&= 4\|\mathcal{A}_{t}(p)\|_{\mathrm{op}}^{2}\|\mathcal{B}_{t}(p)\|^{2}\\
		&\le C_{n}\bigl(tD_{p}(t)+t^{-1}\bigr).
	\end{align*}
	This proves \eqref{eq:Grad_G}.

	Now fix $t>0$ and an open set $U$ on which 	$\lambda_2-\lambda_{1}>\delta>0$. Locally choose a contour $\Gamma$ enclosing only the lowest eigenvalue. The Riesz formula
	\begin{equation*}
		\Pi=\frac{1}{2\pi i}\int_\Gamma (zI-\mathsf{G}^\sharp)^{-1}\dd z
	\end{equation*}
	shows that $\Pi$ is smooth. At a point, choose a normal orthonormal frame whose value is an eigenbasis $\eta,e_2,\ldots,e_n$, with $\eta$ in the lowest eigenspace. For every symmetric endomorphism $V$,
	\begin{equation*}
		D\Pi[V]=\sum_{j=2}^{n} \frac{\langle V\eta,e_{j}\rangle}{\lambda_{1}-\lambda_{j}}
		\bigl(\eta\otimes e_{j}^\flat+e_{j}\otimes\eta^\flat\bigr).
	\end{equation*}
	The summands are orthogonal, and the Hilbert--Schmidt norm of a symmetric endomorphism counts each mixed entry twice. Hence $|D\Pi[V]|\le\delta^{-1}|V|$. Applying this to $V=(\nabla_{X}\mathsf{G})^\sharp$ and summing over $X$ proves \eqref{eq:Grad_{p}i}.
\end{proof}

\begin{corollary}\label{cor:GradG_bound}
	For every $t\ge1$,
	\begin{equation}\label{eq:GradG_bound}
		|\nabla \mathsf{G}(t)|^{2} \le \frac{C_{n}}{t} \mathsf{Q}(t).
	\end{equation}
\end{corollary}
\begin{proof}
	By \autoref{lem:OneUnit}, we have:
	\begin{equation*}
		4t^{2}\int_{M} H_{x}\left|\hes_{x}h_{x}-\frac{\mathsf{G}_{x}(t)}{2t}
		\right|^{2}\dd \Vol_{y}\le \mathsf{Q}_{x}(t)-1
	\end{equation*}
	Therefore, $D_{x}(t)\le (\mathsf{Q}_{x}(t)-1)/(4t^{2})$ and $\mathsf{Q}_{x}(t)\ge 1$, so
	\begin{equation*}
		tD_{x}(t)+t^{-1} \le C(\mathsf{Q}_{x}(t)-1)/t +C/t \le C\mathsf{Q}_{x}(t)/t.
	\end{equation*}
	Substituting this into \eqref{eq:Grad_G} yields the desired bound.
\end{proof}

\subsection{Smooth spectral cutoff}

Fix $0<\varepsilon\le1/2$ and $t\ge6/\varepsilon$, and set
\begin{equation*}
	\begin{aligned}
		\Omega_{\varepsilon}(t)&:=\{x\in M:\mathsf{Q}_{x}(t)<2-\varepsilon\},\\
		U_{\varepsilon}(t)&:=\{x\in M:\lambda_2(x,t)-\lambda_{1}(x,t)>\varepsilon/48\}.
	\end{aligned}
\end{equation*}
On $U_{\varepsilon}(t)$ let $\Pi=\Pi(t)$ be the smooth projection onto
the lowest eigenspace of $\mathsf{G}(t)^{\sharp}$.

We construct the cutoff directly on $\mathcal{V}=\mathrm{Sym}(n)$, equipped
with the Hilbert--Schmidt norm. Put
\begin{equation*}
	\begin{aligned}
		\mathcal{F}_{\varepsilon}
		&:=\{A\in\mathcal{V}:0\le A\le I,\ \lambda_{1}(A)\le\varepsilon/6,
		\ \lambda_2(A)\ge\varepsilon/3\},\\
		\mathcal{O}_{\varepsilon}
		&:=\{A\in\mathcal{V}:\lambda_2(A)-\lambda_{1}(A)>\varepsilon/48\}.
	\end{aligned}
\end{equation*}
Weyl's eigenvalue inequality implies
\begin{equation}\label{eq:matrix-cutoff-separation}
	\operatorname{dist}_{\mathrm{HS}}(\mathcal{F}_{\varepsilon},
	\mathcal{V}\setminus\mathcal{O}_{\varepsilon})\ge\frac{7\varepsilon}{96}.
\end{equation}
Indeed, the spectral gap changes by at most twice the operator-norm distance, whereas it is at least $\varepsilon/6$ on $\mathcal{F}_{\varepsilon}$ and at most $\varepsilon/48$ off $\mathcal{O}_{\varepsilon}$.

Let $d_{\varepsilon}(A)=\operatorname{dist}_{\mathrm{HS}}(A, \mathcal{F}_{\varepsilon})$, put $r=7\varepsilon/768$, and choose an $O(n)$-invariant Lipschitz function $q_{\varepsilon}:\mathcal{V}\to[0,1]$ which equals one for $d_{\varepsilon}\le2r$, vanishes for $d_{\varepsilon}\ge 3r$, and has $\operatorname{Lip}(q_{\varepsilon})\le C/r$. Convolve it with a nonnegative radial mollifier on $\mathcal{V}$ supported in the ball of radius $r$. The resulting function $\Theta_{\varepsilon}\in C^{\infty}(\mathcal{V};[0,1])$ is invariant under orthogonal conjugation and satisfies
\begin{equation}\label{eq:matrix-cutoff-properties}
	\Theta_{\varepsilon}=1\ \text{on }\mathcal{F}_{\varepsilon},
	\qquad \supp\Theta_{\varepsilon}\subset\mathcal{O}_{\varepsilon},
	\qquad |D\Theta_{\varepsilon}|\le\frac{C_n}{\varepsilon}.
\end{equation}
The support inclusion follows from \eqref{eq:matrix-cutoff-separation}, since $4r<7\varepsilon/96$.

Orthogonal invariance makes
\begin{equation*}
	\phi(x):=\Theta_{\varepsilon} \bigl([\mathsf{G}_{x}(t)^{\sharp}]_{\{e_{i}\}}\bigr)
\end{equation*}
independent of the local orthonormal frame $\{e_{i}\}$. Thus $\phi\in C^{\infty}(M; [0,1])$. The projection $\Pi$ itself is defined only on $U_{\varepsilon}(t)$. The combinations $\phi\Pi$ and $\phi\nabla\Pi$ (and hence all weighted scalar expressions used below) extend smoothly by zero to $M$, because $\supp_{M}\phi\subset U_{\varepsilon}(t)$ with a uniform gap buffer.

\begin{lemma}\label{lem:CutOff}
	Let the notation be as above. Then
	\begin{equation*}
		\phi\equiv 1\quad \mathrm{on}\ \Omega_{\varepsilon}(t),\qquad
		\mathrm{supp}_{M}(\phi)\subset U_{\varepsilon}(t),
	\end{equation*}
	and
	\[
		\left\langle \Ric, \Pi \right\rangle \ge \frac{1}{2}\quad \mathrm{on}\ \Omega_{\varepsilon}(t).
	\]
	Moreover, $\Pi$ is smooth on $U_{\varepsilon}(t)$, and
	\begin{equation}\label{eq:cutoff_bound}
		\phi^{2}|\nabla \Pi|^{2} + |\nabla \phi|^{2} \le \frac{C_{n}}{\varepsilon^{2}t}\mathsf{Q}.
	\end{equation}
\end{lemma}
\begin{proof}
	By \autoref{lem:Sub2}, if $x\in\Omega_{\varepsilon}(t)$, then
	\begin{equation*}
		\lambda_{1}(x,t)<\frac{\varepsilon}{6},
		\qquad
		\lambda_{2}(x,t)>\frac{\varepsilon}{3}.
	\end{equation*}
	Thus $\mathsf{G}_{x}(t)^{\sharp}\in\mathcal{F}_{\varepsilon}$, and
	\eqref{eq:matrix-cutoff-properties} together with \autoref{lem:Sub2} gives
	\begin{equation*}
		\phi\equiv 1 \quad\mathrm{and}\quad
		\left\langle\Ric,\Pi\right\rangle\ge\frac{1}{2}
		\qquad\mathrm{on}\ \Omega_{\varepsilon}(t).
	\end{equation*}
	The support assertion is \eqref{eq:matrix-cutoff-properties}. Since $\mathsf{G}(t)$ is smooth and the lowest eigenvalue is simple on $U_{\varepsilon}(t)$, the Riesz projection formula makes $\Pi$ smooth there. The quantitative gap and \eqref{eq:Grad_{p}i} give
	\begin{equation}\label{eq:3:cut:Pi}
		|\nabla\Pi| \le\frac{48}{\varepsilon}|\nabla\mathsf{G}|
		\qquad\mathrm{on}\ U_{\varepsilon}(t).
	\end{equation}
	At a given point, take a local orthonormal frame normal there. Differentiating the frame-independent matrix formula for $\phi$ and using 	\eqref{eq:matrix-cutoff-properties} yields
	\begin{equation}\label{eq:3:cut:phi}
		|\nabla\phi|\le\frac{C_n}{\varepsilon}|\nabla\mathsf{G}|.
	\end{equation}
	
	Since $0\le\phi\le 1$, combining
	\eqref{eq:3:cut:Pi}, \eqref{eq:3:cut:phi}, and \eqref{eq:GradG_bound} gives
	\begin{equation*}
		\phi^{2}|\nabla\Pi|^{2}+|\nabla\phi|^{2}
		\le
		\frac{C_{n}}{\varepsilon^{2}}|\nabla\mathsf{G}|^{2}
		\le
		\frac{C_{n}}{\varepsilon^{2}t}\mathsf{Q}.
	\end{equation*}
	This proves \eqref{eq:cutoff_bound}. 
\end{proof}

\subsection{Weighted Weitzenb\"{o}ck formula}
The first defect unit in \autoref{lem:OneUnit} is pointwise algebra. The analytic input for the second unit is the Weitzenb\"{o}ck formula for one-forms, written in divergence form for a unit one-form. Its closed case is precisely the traced Riccati equation, which explains the underlying geometry.

Let $\eta$ be a smooth unit one-form and put $V=\eta^{\sharp}$. With the convention $\delta\eta=-\diver V$ and $\Delta_{H}=d\delta+\delta d$, the Weitzenb\"{o}ck formula reads
\begin{equation*}
	\Delta_{H}\eta=\nabla^{*}\nabla\eta+\Ric(V,\cdot).
\end{equation*}
Since $|\eta|=1$, we have
\begin{equation*}
	\langle\nabla^{*}\nabla\eta,\eta\rangle=|\nabla\eta|^{2}.
\end{equation*}
On the other hand,
\begin{equation*}
	\begin{aligned}
		\langle d\delta\eta,\eta\rangle
		&=\diver\bigl((\delta\eta)V\bigr)+|\delta\eta|^{2},\\
		\langle\delta d\eta,\eta\rangle
		&=\diver\bigl((\iota_{V}d\eta)^{\sharp}\bigr)+|d\eta|^{2}.
	\end{aligned}
\end{equation*}
Taking the inner product of the Weitzenb\"{o}ck formula with $\eta$ therefore gives the pointwise identity
\begin{equation}\label{eq:unit-oneform-Weitzenbock}
	\begin{aligned}
		\Ric(V,V)=&\diver\left((\iota_{V}d\eta)^{\sharp}+(\delta\eta)V	\right)\\
		&+|d\eta|^{2}+|\delta\eta|^{2}-|\nabla\eta|^{2}.
	\end{aligned}
\end{equation}

The traced Riccati equation is a special case of this identity. Indeed, if $d\eta=0$, then $V$ is geodesic and $\ker\eta$ is locally tangent to a hypersurface foliation. Writing
\begin{equation*}
	B(X,Y)=\langle\nabla_{X}V,Y\rangle,
	\quad X,Y\in\ker\eta, \quad H=\tr_{\ker\eta}B=\diver V=-\delta\eta,
\end{equation*}
we have $|\nabla\eta|^{2}=|B|^{2}$, and \eqref{eq:unit-oneform-Weitzenbock} reduces to
\begin{equation*}
	V(H)+|B|^{2}+\Ric(V,V)=0.
\end{equation*}
For the covariance eigenline, local unit representatives need not be closed and admit no global choice of sign. We therefore retain the full Weitzenb\"{o}ck formula, first in the following integrated form and then in a projection formulation.

\begin{lemma}\label{lem:oneform-Weitzenbock}
	Let $(N, g)$ be an $n$-dimensional Riemannian manifold. For every $\alpha\in C_{c}^{\infty}(T^{*}N)$,
	\begin{equation}\label{eq:integrated-oneform-Weitzenbock}
		\begin{aligned}
			\int_{N}\Ric(\alpha^{\sharp},\alpha^{\sharp})\dd\Vol
			=&\int_{N}\left( |d\alpha|^{2}+|\delta\alpha|^{2}-|\nabla\alpha|^{2} \right)\dd\Vol\\
			\le&(n+1)\int_{N}|\nabla\alpha|^{2}\dd\Vol.
		\end{aligned}
	\end{equation}
	Therefore, if $w>0$ is a smooth function and $\beta\in C_{c}^{\infty}(T^{*}N)$, then
	\begin{equation}\label{eq:weighted-oneform-Weitzenbock}
		\int_{N}w\Ric(\beta^{\sharp},\beta^{\sharp})\dd\Vol
		\le C_{n}\int_{N}w\left(|\nabla\beta|^{2}+|\beta|^{2}|\nabla\log w|^{2} \right)\dd\Vol.
	\end{equation}
\end{lemma}
\begin{proof}
	The equality in \eqref{eq:integrated-oneform-Weitzenbock} is the integrated Weitzenb\"{o}ck formula. Since
	\begin{equation*}
		|d\alpha|^{2}\le2|\nabla\alpha|^{2}, \qquad |\delta\alpha|^{2}\le n|\nabla\alpha|^{2},
	\end{equation*}
	the stated inequality follows. For the weighted estimate, apply \eqref{eq:integrated-oneform-Weitzenbock} to $\alpha=\sqrt{w}\,\beta$ and use
	\begin{equation*}
		\nabla(\sqrt{w}\beta)=\sqrt{w}\left(\nabla\beta+\frac{1}{2}d\log w\otimes\beta 	\right).
	\end{equation*}
	The estimate follows from Young's inequality.
\end{proof}

We now apply this estimate to the lowest covariance eigenline. The projection formulation below is intrinsic and does not require the eigenline to be orientable.

\begin{lemma}[Weighted Weitzenb\"{o}ck estimate]\label{lem:Weight_Wein}
	Let $(M,g)$ be a complete Riemannian manifold, let $U\subset M$ be open, and let $\Pi$ be a smooth rank-one orthogonal projection on $TM|_{U}$. Let $w>0$ and $\phi\in C^{\infty}(M)$ be smooth. Suppose that $\supp_{M}\phi\subset U$, $\int_{M}w\phi^{2}\dd\Vol<\infty$, and $\Ric\ge0$. Then
	\begin{equation}\label{eq:Bochner01}
		\int_{U}w\phi^{2}\langle\Ric,\Pi\rangle\dd\Vol \le C_{n}\int_{U}w\left(
			\phi^{2}|\nabla\Pi|^{2}+|\nabla\phi|^{2} +\phi^{2}|\nabla\log w|^{2}\right)\dd\Vol.
	\end{equation}
\end{lemma}
\begin{proof}
	Set $L=\operatorname{im}\Pi$ and first suppose that $\phi$ has compact support. Let $p:\widetilde{U}\to U$ be the orientation double cover of $L$, equipped with the pullback metric. The line bundle $p^{*}L$ has its tautological smooth unit section $\eta$. All the data below are pulled back to $\widetilde{U}$. Since $p$ is a two-sheeted local isometry, every scalar integral appearing below is twice the corresponding integral on $U$; we	divide by two at the end. We have $p^{*}\Pi=\eta\otimes\eta^{\flat}$ and
	\begin{equation}\label{eq:3:eta_regularity02}
		\nabla_{X}\Pi=(\nabla_{X}\eta)\otimes\eta^{\flat} +\eta\otimes(\nabla_{X}\eta)^{\flat}.
	\end{equation}
	Since $|\eta|=1$, the two terms on the right are orthogonal in the Hilbert--Schmidt norm. Hence
	\begin{equation*}
		|\nabla\Pi|^{2}=2|\nabla\eta|^{2}.
	\end{equation*}
	Apply \eqref{eq:weighted-oneform-Weitzenbock} on $\widetilde{U}$ to the compactly supported one-form $\beta=\phi\eta^{\flat}$. The two terms in
	\begin{equation*}
		\nabla\beta=d\phi\otimes\eta^{\flat} +\phi\nabla\eta^{\flat}
	\end{equation*}
	are orthogonal. Consequently,
	\begin{equation*}
		|\nabla\beta|^{2} =|\nabla\phi|^{2}+\phi^{2}|\nabla\eta|^{2} =|\nabla\phi|^{2}+\frac{1}{2}\phi^{2}|\nabla\Pi|^{2},
	\end{equation*}
	whereas
	\begin{equation*}
		\Ric(\beta^{\sharp},\beta^{\sharp}) 	=\phi^{2}\langle\Ric,\Pi\rangle.
	\end{equation*}
	Dividing the resulting inequality by two proves \eqref{eq:Bochner01} for compactly supported $\phi$, without any orientability assumption on $L$.

	For general $\phi$, the classical Greene--Wu exhaustion theorem \cite[Corollary to Proposition 2.1]{GW1979} provides a smooth proper exhaustion function $\rho\ge 0$ with $|\nabla \rho|<1$. Put $\chi_{R}=\chi(\rho/R)$, where $\chi\in C_{c}^{\infty}([0,\infty))$ equals one on $[0,1]$ and vanishes on $[2,\infty)$. Then
	\begin{equation*}
		0\le\chi_{R}\le1,\quad \chi_{R}\longrightarrow1, \quad \supp\chi_{R}\Subset M, \quad |\nabla\chi_{R}|\le\frac{C}{R}.
	\end{equation*}
	The condition $\supp_{M}\phi\subset U$ ensures that $\chi_{R}\phi\in C_{c}^{\infty}(U)$. Apply the compact-support estimate to $\chi_{R}\phi$. Its only new error satisfies
	\begin{equation*}
		\int_{M}w\phi^{2}|\nabla\chi_{R}|^{2}\dd\Vol \le\frac{C}{R^{2}}\int_{M}w\phi^{2}\dd\Vol\longrightarrow0.
	\end{equation*}
	If the right-hand side of \eqref{eq:Bochner01} is finite, dominated convergence controls the terms containing $\chi_{R}^{2}$, while
	\begin{equation*}
		|\nabla(\chi_{R}\phi)|^{2} \le2\chi_{R}^{2}|\nabla\phi|^{2} +2\phi^{2}|\nabla\chi_{R}|^{2}.
	\end{equation*}
	Fatou's lemma applies to the left-hand side because $\Ric\ge 0$. Absorbing the harmless factor two into $C_{n}$ proves the result. If the right-hand side is infinite, the assertion is automatic.
\end{proof}

\begin{remark}
	The compact-support argument does not use $\Ric\ge 0$; that hypothesis enters only when the exhaustion is removed.
\end{remark}

\begin{lemma}[\cite{Ni2004a}, \cite{Col2012}]\label{lem:Fisher_Bound}
	For every $z\in M$ and $s>0$,
	\begin{equation}\label{eq:Fisher_Bound}
		\int_{M} H(z,y,s)|\nabla_{y} \log H(z,y,s)|^{2} \dd\Vol_{y}\le \frac{n}{2s}.
	\end{equation}
\end{lemma}
\begin{proof}
	This is the integrated Li--Yau estimate. See \cite[Proposition 1.1]{Ni2004a}. See also Colding \cite[\S5.1]{Col2012}.
\end{proof}

\subsection{Proof of the two-unit source}

\begin{proof}[Proof of \autoref{thm:two-unit}]
	Fix $t\ge 1$, $s>0$, and $z\in M$, and set
	\begin{equation*}
		Y=P_{s}(\mathsf{Q}_{\bullet}(t))(z).
	\end{equation*}
	If $Y=+\infty$, then there is nothing to prove.

	For $0<\varepsilon\le1/2$ with $t\ge6/\varepsilon$, note first that $0\le\phi\le1$ and stochastic completeness give $\int_{M}H\phi^{2}\dd\Vol\le1$, so the integrability hypothesis in \autoref{lem:Weight_Wein} holds. Apply that lemma with $w(x)=H(z,x,s)$, $U=U_\varepsilon(t)$, and $\Pi=\Pi(t)|_{U_\varepsilon(t)}$. Since $\phi=1$ and $\langle\Ric,\Pi\rangle\ge1/2$ on $\Omega_\varepsilon(t)$, we have
	\begin{equation*}
		\frac{1}{2} P_{s}( \mathbf{1}_{\Omega_{\varepsilon}(t)}) (z) \le \int_{M} w \phi^{2} \left\langle \Ric, \Pi \right\rangle \dd\Vol.
	\end{equation*}
	The cutoff energy estimate and the global Fisher bound \eqref{eq:Fisher_Bound} give
	\begin{equation*}
		\int_{M} w (\phi^{2}|\nabla\Pi|^{2} +|\nabla \phi|^{2})\dd\Vol \le \frac{C_{n}}{\varepsilon^{2}t} Y,
	\end{equation*}
	and 
	\begin{equation*}
		\int_{M} w\phi^{2}|\nabla \log w|^{2} \dd\Vol \le \frac{n}{2s}.
	\end{equation*}
	They imply
	\begin{equation}\label{eq:two-unit-01}
		P_{s}(\mathbf{1}_{\Omega_{\varepsilon}(t)})(z) \le A_{n} \varepsilon^{-2} \left( \frac{Y}{t} + \frac{1}{s} \right).
	\end{equation}
	Here $A_{n}<\infty$ depends only on $n$. The Fisher term is placed under the common factor $\varepsilon^{-2}$ using $\varepsilon^{-2}\ge 1$. The derivative of $w$ is in the output variable, as in \autoref{lem:Fisher_Bound}.
	Since $\mathsf{Q}\ge 2-\varepsilon$ away from $\Omega_{\varepsilon}(t)$, stochastic completeness of $M$ gives 
	\begin{equation*}
		Y\ge m+ (2-\varepsilon)(1-m)=2-\varepsilon-(1-\varepsilon)m,
	\end{equation*}
	where $m:=P_{s}(\mathbf{1}_{\Omega_{\varepsilon}(t)})(z)$. Together with \eqref{eq:two-unit-01} this yields
	\begin{equation*}
		Y\ge 2-\varepsilon- A_{n}(1-\varepsilon)\varepsilon^{-2}\left( \frac{Y}{t} + \frac{1}{s} \right).
	\end{equation*}
	Set $\rho=1/t+1/s$. If
	\[
		\rho \le \rho_{0}:= 6^{-3/2},
	\]
	take $\varepsilon=\rho^{1/3}$. Then
	\begin{equation*}
		\varepsilon \le 6^{-1/2}< 1/2,\quad
		t\ge \rho^{-1}\ge 6\rho^{-1/3}=6/\varepsilon.
	\end{equation*}
	Write
	\begin{equation*}
		\beta=A_{n}(1-\varepsilon)\varepsilon^{-2}t^{-1}, \quad
		\theta=A_{n}(1-\varepsilon)\varepsilon^{-2}s^{-1}.
	\end{equation*}
	Then the preceding inequality can be rewritten as
	\begin{equation*}
		Y\ge \frac{2-\varepsilon-\theta}{1+\beta}
		= 2-\frac{\varepsilon+\theta+2\beta}{1+\beta}.
	\end{equation*}
	Since both $\theta$ and $\beta$ are at most $A_n\rho^{1/3}$, we obtain
	\[
		Y\ge 2-(1+3A_n)\rho^{1/3}.
	\]
	If $\rho>\rho_{0}$, we can enlarge $C_n$ so that the claim follows from $Y\ge P_s(\mathbf{1})(z)=1$, which holds by \autoref{lem:OneUnit} and stochastic completeness. This proves \eqref{eq:two-unit}.
\end{proof}

\section{Nash entropy}
The Nash entropy is
\begin{equation}\label{eq:Nash}
	S_{x}(t)=-\int_{M}H(x,y,t)\log H(x,y,t)\dd \Vol_{y} - \frac{n}{2}\log(4\pi t)- \frac{n}{2}.
\end{equation}

We record the following standard tail estimate in the locally uniform form needed for pole-variable differentiation. Its proof uses only the Li--Yau estimate and Bishop--Gromov comparison and does not assume bounded geometry.
\begin{lemma}\label{lem:AbsoluteEntropyTails}
	Fix $o\in M$. For every compact subset $K\subset M$, every $0<a<b<\infty$, and every integer $N\ge0$, 
	\begin{equation}\label{eq:AET01}
		\sup_{(x,t)\in K\times [a,b]}\int_{M} 
		H(x,y,t)\bigl(1+|\log H(x,y,t)|\bigr)
		\bigl(1+d(o,y)\bigr)^{N}\dd\Vol_{y} <\infty.
	\end{equation}
	If $0\le \phi_{m}\le 1$ and $\phi_{m}\to 1$ are compactly supported output cutoff functions, then 
	\begin{equation}\label{eq:AET02}
		\int_{M} (1-\phi_{m}) |\log H(x,y,t)| \dd\mu_{x,t}(y)\to 0\qquad \text{in}\ L^{1}(K\times (a,b)).
	\end{equation}
\end{lemma}
\begin{proof}
	First, by Li--Yau's Gaussian upper estimate \cite[Corollary~3.1]{LY1986} and Bishop--Gromov volume comparison, we have:
	\begin{equation}\label{eq:entropy-Gaussian}
		H(x, y, t)\le \frac{C_{n}}{\Vol B(x, \sqrt{t})} \exp\left( -\frac{d(x, y)^{2}}{5t} \right).
	\end{equation}
	Compactness of $K$ together with $t\ge a$ gives:
	\begin{equation*}
		v_{0}:=\inf_{x\in K} \Vol B(x, \sqrt{a}) >0.
	\end{equation*}
	It follows that, after increasing $A\ge 1$,
	\begin{equation*}
		H(x, y, t) \le A \exp \left( -\frac{d(x, y)^{2}}{5b} \right)\qquad (x, t)\in K\times [a,b].
	\end{equation*}
	Using 
	\[
		r|\log r| \le (\log A)r+\frac{2}{e}r^{1/2} \qquad 0<r\le A,
	\]
	and fixing $o\in M$, the triangle inequality gives
	\[
	H(x,y,t)\bigl(1+|\log H(x,y,t)|\bigr) \le C_{K,a,b} \exp\left(-\frac{d(o,y)^{2}}{20b}\right).
	\]
	Every polynomial multiple of the right-hand side belongs to $L^1(M)$ by Bishop--Gromov volume comparison. This proves \eqref{eq:AET01}, while dominated convergence on $K\times(a,b)\times M$ proves \eqref{eq:AET02}.
\end{proof}

\begin{corollary}\label{cor:CanonicalRep}
	The quantities $S$, $\mathsf{E}$, and $\mathsf{Q}$ are finite and smooth on $M\times(0,\infty)$.
\end{corollary}
\begin{proof}
	The assertions for $\mathsf{E}$ and $\mathsf{Q}$ follow from \autoref{prop:HeatTensorRegularity}. By induction using the product rule, every pole-space-time derivative of $H\log H$ is $H$ times a polynomial in the relative heat jets, with at most one factor $1+\log H$. The heat-jet estimates of \autoref{lem:heat-jets} therefore bound it on compact positive-time cylinders by $H(1+|\log H|)$ times a polynomial in $d(o,y)$. Thus \autoref{lem:AbsoluteEntropyTails} permits arbitrary differentiation under the entropy integral and proves the assertion for $S$.
\end{proof}

We use the Nash entropy $S_{x}(t)$ defined in \eqref{eq:Nash}, following \cite{Ni2004b} with the normalization used by Colding \cite[\S 5.1]{Col2012}. The pole-variable heat-operator formula appears in the calculation in the proof of \cite[Theorem~5.9]{Bam2020}. We record its static, complete-noncompact version. \autoref{lem:AbsoluteEntropyTails}, \autoref{prop:Gbounds}, and \autoref{lem:heat-jets} justify the required output-variable integration. The next two propositions use the heat-kernel variables differently: the first differentiates the pole $x$, while in the second we fix $x$ and integrate by parts in the output variable $y$.

\begin{proposition}\label{prop:PoleEntropy}
	Pointwise on $M\times(0,\infty)$,
	\begin{equation}\label{eq:PoleEntropy01}
		\square_{x} S_{x}(t) = -\frac{\mathsf{E}_{x}(t)}{t^{2}}.
	\end{equation}
\end{proposition}
\begin{proof}
	For a positive function $u$ with $\square_{x} u=0$, direct calculation shows:
	\begin{equation}\label{eq:PoleEntropy02}
		\square_{x}(-u\log u) = \frac{|\nabla_{x}u|^{2}}{u}.
	\end{equation}
	Apply this to $H(x, y, t)$ for fixed $y$ and then integrate in $y$. Differentiation under the integral is justified here: \autoref{lem:AbsoluteEntropyTails} controls the term containing $H |\log H|$, while \autoref{prop:Gbounds} gives
	\begin{equation*}
		\int_{M} H_{x} |d_{x}h_{x}|^{2} \dd\Vol_{y} = \frac{\tr \mathsf{G}_{x}(t)}{2t} \le \frac{n}{2t},
	\end{equation*}
	which is integrable on every compact positive-time pole-variable cylinder. \autoref{lem:heat-jets} and \autoref{lem:AbsoluteEntropyTails} justify the classical differentiation under the integral. We obtain
	\begin{equation*}
		\square_{x} \left( -\int_{M} H_{x }\log H_{x} \dd\Vol_{y} \right) = \frac{\tr \mathsf{G}_{x}(t)}{2t}.
	\end{equation*}
	Subtracting the time derivative of $\frac{n}{2}\log(4\pi t)+\frac{n}{2}$ proves \eqref{eq:PoleEntropy01}.
\end{proof}

\begin{proposition}\label{prop:EntropyMono}
	For every fixed $x$, the function $S_{x}$ is locally absolutely continuous and nonincreasing on $(0, \infty)$. Moreover 
	\begin{equation*}
		\lim_{t\to 0^{+}} S_{x}(t) = 0,
	\end{equation*}
	therefore $S_{x}(t)\le 0$.
\end{proposition}
\begin{proof}
	For a fixed $x$, Colding's entropy $S$ agrees with $S_{x}$ defined in \eqref{eq:Nash}. His entropy identity gives
	\[
		tS_{x}'(t) = t\int_{M} H_{x}|d_{y}h_{x}|^{2} \dd\Vol_{y}-\frac{n}{2}\le 0,
	\]
	see \cite[Lemma~5.1]{Col2012}. The inequality is the integrated Li--Yau estimate. The corresponding integration by parts on complete noncompact manifolds with $\Ric \ge 0$ is justified in the section ``Extensions and the value of $\mu(0)$'' of \cite{Ni2004b}. The entropy-tail and Fisher-information bounds established above also verify the required local integrability. Thus $S_{x}\in AC_{\mathrm{loc}}(0,\infty)$ and is nonincreasing.
	The limit $\lim_{t\to 0^{+}}S_{x}(t)=0$ is \cite[Proposition~1.1(ii)]{Ni2004a}; see also \cite{Xu2013}. The monotonicity of $S_{x}$ then implies $S_{x}(t)\le 0$.
\end{proof}

\section{Entropy decay and volume growth}

\begin{lemma}\label{lem:MiniPotential}
	Let $0<a<T$. Suppose that $u,f$ are smooth on $M\times(a,T]$, that
	$u$ is continuous on $M\times[a,T]$, and that $u\ge0$, $f\ge0$, and
	$\square u\ge f$. Then
	\begin{equation}\label{eq:MiniPotential01}
		u(x, T)\ge \int_{a}^{T} P_{T-t}(f(\bullet, t))(x)\dd t.
	\end{equation}
\end{lemma}
\begin{proof}
	This is the standard minimality principle for an inhomogeneous heat equation; for the Dirichlet heat semigroup and the exhaustion construction of the minimal heat kernel, see Grigor'yan \cite[Chapters~5, 7, and~8]{Gri2009}. On the spacetime manifold $M\times(a,T)$ choose a smooth proper exhaustion $\rho$ and a smooth nonincreasing function $\chi$ that is one on $(-\infty,1]$ and zero on $[2,\infty)$. Then $\eta_{j}=\chi(\rho/j)$ satisfies $0\le\eta_{j}\uparrow1$ and has compact support; in particular, its spatial projection is compact. Choose relatively compact smooth domains $\Omega_{j}\uparrow M$ containing these spatial projections, and set $f_{j}=\eta_{j}f$. Thus $0\le f_{j}\uparrow f$. With $P_{r}^{\Omega_{j}}$ denoting the Dirichlet heat semigroup, set
	\[
		V_{j}(x,s)=\int_a^sP_{s-t}^{\Omega_{j}}\bigl(f_{j}(\cdot,t)\bigr)(x)\dd t.
	\]
	Then $V_{j}$ is a classical solution of $\square V_{j}=f_{j}$ with zero parabolic boundary values. Since $f_{j}\le f$, the classical parabolic maximum principle applied to $V_{j}-u$ gives $0\le V_{j}\le u$. Domain monotonicity of the Dirichlet heat kernels, their convergence to the minimal heat kernel, and monotone convergence yield
	\[
		V_{j}(x,T)\uparrow\int_a^TP_{T-t}\bigl(f(\cdot,t)\bigr)(x)\dd t.
	\]
	Passing to the limit proves \eqref{eq:MiniPotential01}. The global exhaustion is essential: without a boundary condition at infinity, a nonnegative homogeneous caloric remainder need not vanish.
\end{proof}

Combining the new two-unit estimate for the centered source with the preceding classical source identities yields the following new entropy decay:

\begin{proposition}\label{prop:TwoDeftoEntropDecay}
	For every $x\in M$ and $T\ge 2$, 
	\begin{equation}\label{eq:TwoDeftoEntropDecay01}
		S_{x}(T) \le - \log T + C_{n}.
	\end{equation}
\end{proposition}
\begin{proof}
	By \autoref{prop:EntropyMono} and \eqref{eq:2:E_Bound}, the functions $-S$ and $\mathsf{E}$ are nonnegative. \autoref{cor:CanonicalRep} makes both finite and smooth on every positive time. Both applications of \autoref{lem:MiniPotential} begin at time $1$.
	First, $\square_{x}\mathsf{E}=\mathsf{Q}/2$ gives, for $t>1$,
	\begin{equation}\label{eq:Growth_E}
		\mathsf{E}_{x}(t)\ge \frac{1}{2}\int_{1}^{t}P_{t-\tau}(\mathsf{Q}_{\bullet}(\tau))(x)\dd\tau.
	\end{equation}
	Second \autoref{prop:PoleEntropy} gives $\square_{x}(-S)=\mathsf{E}/t^{2}$, therefore:
	\begin{equation}\label{eq:Growth_S}
		-S_{x}(T)\ge \int_{1}^{T}\frac{1}{t^{2}}P_{T-t}(\mathsf{E}_{\bullet}(t))(x)\dd t.
	\end{equation}
	For each $1<t<T$, Tonelli's theorem and the semigroup property applied to \eqref{eq:Growth_E} give
	\begin{equation*}
		P_{T-t}(\mathsf{E}_{\bullet}(t))(x)\ge \frac{1}{2}\int_{1}^{t}
		P_{T-\tau}(\mathsf{Q}_{\bullet}(\tau))(x)\dd\tau.
	\end{equation*}
	Plugging this into \eqref{eq:Growth_S} and applying Tonelli once more, we obtain:
	\begin{equation}\label{eq:Whole_S_Growth}
		\begin{aligned}
		 -S_{x}(T)&\ge \frac{1}{2}\int_{1}^{T}\int_{1}^{t}\frac{1}{t^{2}}
		 P_{T-\tau}(\mathsf{Q}_{\bullet}(\tau))(x)\dd\tau \dd t\\
		 &= \frac{1}{2}\int_{1}^{T}\left(\int_{\tau}^{T}\frac{\dd t}{t^{2}}\right)
		 P_{T-\tau}(\mathsf{Q}_{\bullet}(\tau))(x)\dd\tau\\
		 &=\frac{1}{2}\int_{1}^{T}\left( \frac{1}{\tau}-\frac{1}{T} \right)
		 P_{T-\tau}(\mathsf{Q}_{\bullet}(\tau))(x)\dd\tau.
		\end{aligned}
	\end{equation}
	All integrals above are nonnegative; hence Tonelli applies without an integrability assumption. For almost every $\tau\in (1, T)$, both $\tau$ and $T-\tau$ are positive, so \autoref{thm:two-unit} applies:
	\begin{equation*}
		P_{T-\tau}(\mathsf{Q}_{\bullet}(\tau))(x)\ge 2- C_{n}
		\left( \tau^{-1} + (T-\tau)^{-1} \right)^{1/3}.
	\end{equation*}
	Therefore
	\begin{equation}\label{eq:S_Key}
		-S_{x}(T)\ge \log T-1+\frac{1}{T}-\frac{C_{n}}{2}f(T),
	\end{equation}
	where
	\begin{equation*}
		f(T)=\int_{1}^{T}\left( \frac{1}{\tau} - \frac{1}{T} \right)
		\left( \tau^{-1} + \left( T-\tau \right)^{-1} \right)^{1/3} \dd\tau.
	\end{equation*}
	Since $0\le \tau^{-1}-T^{-1} \le \tau^{-1}$, the subadditivity of $r\to r^{1/3}$ gives:
	\begin{equation*}
		\int_{1}^{T}\left( \frac{1}{\tau}-\frac{1}{T} \right)\tau^{-1/3}
		\le \int_{1}^{T} \tau^{-4/3} \dd\tau\le 3.
	\end{equation*}
	We split the integral at $T/2$:
	\begin{equation*}
		\int_{1}^{T/2}\tau^{-1}(T-\tau)^{-1/3} \dd\tau \le C T^{-1/3}\log T\le C
	\end{equation*}
	\begin{equation*}
		\int_{T/2}^{T}\tau^{-1}(T-\tau)^{-1/3}\dd\tau \le \frac{2}{T}\int_{0}^{T/2}
		r^{-1/3}\dd r\le C.
	\end{equation*}
	Both estimates are uniform in $T\ge 2$; therefore $f(T)\le C$, where $C<\infty$ is universal. After enlarging $C_n$, \eqref{eq:S_Key} implies \eqref{eq:TwoDeftoEntropDecay01}.
\end{proof}

\begin{proof}[Proof of \autoref{thm:Main}]
	The Li--Yau Gaussian upper estimate \eqref{eq:entropy-Gaussian} implies that
	\begin{equation*}
		H(x, y, t)\le \frac{C_{n}}{\Vol B(x, \sqrt{t})} \exp \left( -\frac{d(x, y)^{2}}{5t} \right).
	\end{equation*}
	Take $-\log$, multiply by $H$, and integrate first over $B(x,R)$. The constant term on the right is multiplied by $\int_{B(x,R)}H\dd\Vol_{y}$. Letting $R\to\infty$ is justified on the left by \autoref{lem:AbsoluteEntropyTails}; stochastic completeness handles the constant term, and monotone convergence applies to the nonnegative distance term. Discarding that term gives
	\begin{equation*}
		-\int_{M} H \log H\dd\Vol_{y} \ge \log \Vol B(x, \sqrt{t}) - C_{n},
	\end{equation*}
	and therefore after absorbing the fixed normalization constants, we have
	\begin{equation*}
		\log \frac{\Vol(B(x, \sqrt{t}))}{t^{n/2}} \le S_{x}(t) + C_{n}.
	\end{equation*}
	\autoref{prop:TwoDeftoEntropDecay} yields 
	\begin{equation*}
		\Vol(B(x, \sqrt{t})) \le C_{n} t^{(n-2)/2}, \quad (t\ge 2).
	\end{equation*}
	For $R\ge \sqrt{2}$, take $t=R^{2}$. For $R\le \sqrt{2}$, Bishop--Gromov gives:
	\begin{equation*}
		\Vol(B(x, R))\le \omega_{n}R^{n} \le 2\omega_{n} R^{n-2}.
	\end{equation*}
	This completes the proof.
\end{proof}

\appendix
\section{Pole vs output identities}

In this appendix, we compare the identities defined by the pole-variable and the output-variable. For a fixed pole, Ni and Colding differentiate in the output variable and obtain the usual entropy monotonicity formulas \cite{Ni2004b,Col2012}. On the other hand, Ni's pointwise $W$-density identity \cite[Lemma~2.2]{Ni2004b} may be applied to $x\mapsto H(x, y, t)$ for fixed $y$ and then integrated in $y$; this gives the raw pole-variable $W$ and $\mathsf{E}$ source identities below. Bamler explicitly computes the pole-variable heat-operator formula for pointed Nash entropy in the Ricci-flow setting \cite[proof of Theorem~5.9]{Bam2020}.

The scalar combination defining the defect $\mathsf{E}$ is algebraically implicit in these papers. The purpose of this appendix is to distinguish the output variable and pole variable identities and to record their exact relation to the centered source $\mathsf{Q}$. 

Write the heat kernel as:
\begin{equation*}
	H(x, y, t)=(4\pi t)^{-n/2}e^{-f_{x}(y, t)}.
\end{equation*}
Equivalently,
\begin{equation*}
	f_{x}(y, t)=h_{x}(y, t)-\frac{n}{2}\log(4\pi t),
\end{equation*}
so the pole and output variable derivatives of $f_{x}$ agree with the corresponding derivatives of $h_{x}$. Fixing the pole variable $x$ and differentiating in the output variable $y$, define the output Fisher-information term
\begin{equation*}
	F_{x}^{\mathrm{out}}(t):=t\int_{M}H_{x}|d_{y}f_{x}|^{2}\dd\Vol_{y}-\frac{n}{2},
	\qquad W_{x}^{\mathrm{out}}:=S_{x}+F_{x}^{\mathrm{out}}.
\end{equation*}

The fixed-pole calculations of Ni and Colding, see \cite{Ni2004b} and \cite[Lemmas~5.1 and~5.2]{Col2012}, give
\begin{align}
	\partial_{t}S_{x}&=\frac{F_{x}^{\mathrm{out}}}{t}, \label{eq:appendix-output-S}\\
	\partial_{t}\left(tF_{x}^{\mathrm{out}}\right)&=-2t^{2}\int_{M}H_{x}\left(\left|\hes_{y}f_{x}-\frac{g_{y}}{2t}\right|^{2}
	+\Ric_{y}(\nabla_{y}f_{x},\nabla_{y}f_{x})\right)\dd\Vol_{y},\label{eq:appendix-output-tF}\\
	\partial_{t}W_{x}^{\mathrm{out}}&=-2t\int_{M}H_{x}\left(\left|\hes_{y}f_{x}-\frac{g_{y}}{2t}\right|^{2}
	+\Ric_{y}(\nabla_{y}f_{x},\nabla_{y}f_{x}) \right)\dd\Vol_{y}. \label{eq:appendix-output-W}
\end{align}

For comparison with these fixed-pole output identities, we introduce the normalized pole Fisher-information term and its associated $W$-combination:
\begin{equation}\label{eq:appendix-pole-FW-def}
	\begin{aligned}
		F_{x}^{\mathrm{pole}}(t)&:=t\int_{M}H_{x}|d_{x}f_{x}|^{2}\dd\Vol_{y}-\frac{n}{2}
		=-\frac{\mathsf{E}_{x}(t)}{t},\\
		\Wpole_{x}(t)&:=S_{x}(t)+F_{x}^{\mathrm{pole}}(t)	=S_{x}(t)-\frac{\mathsf{E}_{x}(t)}{t}.
	\end{aligned}
\end{equation}
Here the last equality in the first line follows from
\begin{equation*}
	2t\int_{M}H_{x}|d_{x}f_{x}|^{2}\dd\Vol_{y} =\tr_g\mathsf{G}_{x}(t),
\end{equation*}
together with the definition of $\mathsf{E}$.

The pointwise calculation in \eqref{eq:Bochner04}, justified by \autoref{prop:PolewiseDefectSource}, identifies the raw pole source with the centered source \eqref{eq:Q}:
\begin{equation}\label{eq:appendix-raw-centered-Q}
	\mathsf{Q}_{x}(t)=4t^{2}\int_{M} H_{x}\left(\left|\hes_{x} f_{x}-\frac{g_{x}}{2t}\right|^{2}+\Ric_{x}(\nabla_{x} f_{x},\nabla_{x} f_{x})\right)\dd\Vol_{y} .
\end{equation}
Thus the raw and centered expressions define the same function pointwise on $M\times(0,\infty)$; it is smooth by \autoref{prop:HeatTensorRegularity}. Together with \autoref{prop:PoleEntropy}, this gives the following identities pointwise, and hence also in $\mathcal D'(M\times(0,\infty))$:

\begin{align}
	\square_{x}S&=\frac{F^{\mathrm{pole}}}{t}, \label{eq:appendix-pole-S}\\
	\square_{x}\left(tF^{\mathrm{pole}}\right)&=-\frac{\mathsf{Q}}{2}, \label{eq:appendix-pole-tF}\\
	\square_{x}\Wpole&=-\frac{\mathsf{Q}}{2t}.\label{eq:appendix-pole-W}
\end{align}
Indeed, the first line is \autoref{prop:PoleEntropy}, because $F^{\mathrm{pole}}=-\mathsf{E}/t$, and the second follows from $ tF^{\mathrm{pole}} =-\mathsf{E}$ and \autoref{prop:PolewiseDefectSource}. Finally,
\begin{equation*}
	\square_{x}\left(tF^{\mathrm{pole}}\right) =F^{\mathrm{pole}}+t\square_{x}F^{\mathrm{pole}},
\end{equation*}
so the first two lines give the third after adding $\square_{x}S$ and $\square_{x}F^{\mathrm{pole}}$. By the pointwise identity \eqref{eq:appendix-raw-centered-Q}, the last identity is equivalently
\begin{equation}\label{eq:appendix-pole-dissipation}
	\square_{x}\Wpole =-2t\int_{M}H_{x}\left(\left|\hes_{x}f_{x}-\frac{g_{x}}{2t}\right|^{2}
	+\Ric_{x}(\nabla_{x}f_{x},\nabla_{x}f_{x})\right)\dd\Vol_{y}
\end{equation}
pointwise; its right-hand side is the smooth density $-\mathsf{Q}/(2t)$.

The proof of the main theorem does not use $F^{\mathrm{pole}}$ or $\Wpole$. It propagates $\mathsf{Q}$ first through $\square_{x}\mathsf{E}=\mathsf{Q}/2$, and then propagates $\mathsf{E}/t^{2}$ through $\square_{x}(-S)=\mathsf{E}/t^{2}$. Equation \eqref{eq:appendix-pole-W} algebraically combines the two source equations and makes the relation with Ni's pointwise density identity transparent. It also gives an equivalent compressed route through a single minimal-potential comparison. Indeed, $S\le0$ and $\mathsf{E}\ge0$ imply $\Wpole \le 0$, so \autoref{lem:MiniPotential} gives, for $T\ge 2$,
\begin{equation*}
	-\Wpole_{x}(T) \ge\frac{1}{2}\int_{1}^{T}\frac{1}{\tau}
	P_{T-\tau}[\mathsf{Q}_{\bullet}(\tau)](x)\dd\tau.
\end{equation*}
Since $\mathsf{E}_{x}(T)/T\le n/2$, this yields
\begin{equation*}
	-S_{x}(T)=-\Wpole_{x}(T)-\frac{\mathsf{E}_{x}(T)}{T}
	\ge\frac{1}{2}\int_{1}^{T}\frac{1}{\tau}
	P_{T-\tau}[\mathsf{Q}_{\bullet}(\tau)](x)\dd\tau-\frac{n}{2}.
\end{equation*}
The main proof retains the two-step route because it displays the separate propagation through $\mathsf{E}$ and $-S$; it does not use $F^{\mathrm{pole}}$ or $\Wpole$.

The pole identity is an integrated consequence of Ni's pointwise identity, once the noncompact passage is justified as in \autoref{prop:PolewiseDefectSource}. Heat-kernel symmetry exchanges $x$ and $y$ before either variable is frozen or integrated; it does not convert a fixed-pole output identity into a pole-variable identity after integration. In particular, \eqref{eq:appendix-output-W} is an ordinary time-monotonicity formula for a fixed pole, whereas \eqref{eq:appendix-pole-W} contains the pole Laplacian. Thus it does not imply, and we do not assert or use, monotonicity of $t\mapsto\Wpole_{x}(t)$ for fixed $x$.

\section{Positive-time regularity of the heat quantities}
\label{app:heat-regularity}
This appendix proves \autoref{prop:HeatTensorRegularity}. Its only analytic input is a locally uniform, Gaussian-integrable bound for pole-variable heat-kernel jets.

\subsection{Uniform pole-variable heat-kernel jets}

Fix a base point $o\in M$ and put $R(y)=1+d(o,y)$. The next lemma is the tail statement needed throughout the paper. We use $\nabla_{x}^m$ for an arbitrary $m$-th covariant derivative in the pole variable.

\begin{lemma}\label{lem:heat-jets}
	Let $K\Subset M$ and $0<a<b<\infty$. For every $m,\ell\ge0$ there are constants $C,N<\infty$, depending on $K,[a,b],m,\ell$ and the metric on a fixed neighborhood of $K$, such that
	\begin{equation}\label{eq:rel_heat}
		\frac{|\nabla_{x}^m\partial_{t}^\ell H(x, y, t)|}{H(x, y, t)}
		\le C R(y)^{N}, \qquad (x, t)\in K\times[a, b], \quad y\in M.
	\end{equation}
	Moreover, for every $q\ge0$ there exist $C_{q}, c_{q}>0$ such that
	\begin{equation}\label{eq:heat_maj}
		\sup_{(x, t)\in K\times[a,b]}
		 H(x, y, t)R(y)^{q} \le C_{q}\exp\bigl(-c_{q} d(o, y)^{2}\bigr).
	\end{equation}
	The right-hand side is integrable over $M$. Thus every finite product of relative pole-time jets, multiplied by $H$, has a locally uniform integrable majorant in the output variable.
\end{lemma}
\begin{proof}
	Our goal is to find bounds on all pole-time derivatives over the fixed cylinder $K\times[a,b]$ whose dependence on the output point $y$ is at most polynomial. The proof has four steps: First, Kotschwar's estimate controls the spatial oscillation of the heat kernel on a fixed time slice. Second, the Li--Yau Harnack inequality then controls its backward time oscillation. These two estimates bound the heat kernel on a shrinking backward parabolic cylinder by its value at the top point. Third, higher-order interior parabolic regularity converts this cylinder bound into derivative bounds. Finally, the Li--Yau Gaussian upper estimate supplies an integrable tail in the output variable.

	Kotschwar's heat-kernel estimate \cite[Theorem~1, equation~(2)]{Kot2007}, in the first (pole) variable, gives, for any fixed $\delta\in(0,4)$,
	\begin{equation}\label{eq:K-gradient}
		|\nabla_{x}\log H(x, y, s)|^2 \le \frac{2}{s}\left(C_{n,\delta}+ \frac{d(x, y)^{2}}{(4-\delta)s}\right).
	\end{equation}
	The remaining argument is local in the pole. Cover the original compact set by finitely many compact sets $K_{j}$ for which $K_{j}\Subset V_{j}\Subset U_{j}$, with each $U_{j}$ relatively compact in a coordinate chart. It is enough to prove the estimate on one such triple, which we relabel as $K\Subset V\Subset U$, and then take the maximum of the finitely many constants. Choose $r_{0}>0$ so that $\overline{B(x, 4r_{0})}\Subset V$ for every $x\in K$. Decrease $r_{0}$ so that $r_{0}\le 1$ and $4r_{0}^{2}< a/2$, and set
	\begin{equation*}
		r_{y}=\frac{r_{0}}{R(y)}.
	\end{equation*}
	For $(x, t)\in K\times[a, b]$, introduce the backward cylinder
	\begin{equation*}
		\mathcal{C}_{x, y, t} =B(x, 2r_{y})\times[t-4r_{y}^{2}, t]
	\end{equation*}
	which is contained in $V\times[a/2,b]$. Moreover, every $(z, s)\in\mathcal{C}_{x, y, t}$ satisfies $t/2<s\le t$. In summary,
	\begin{equation}\label{eq:cylinder-location}
		\mathcal C_{x,y,t}\subset V\times[a/2,b], \qquad t/2<s\le t \quad\text{for every }(z, s)\in\mathcal C_{x, y, t}.
	\end{equation}

	For $z\in B(x, 2r_{y})$, integration of \eqref{eq:K-gradient} along a minimizing curve from $x$ to $z$ shows, uniformly for $s\in[a/2, b]$, that
	\begin{equation*}
		H(z, y, s)\le C H(x,y,s).
	\end{equation*}
	Indeed, $d(w, y)\le d(o, y)+C_U$ for every $w\in B(x,2r_{y})$, so $|\nabla_{x} \log H|\le C R(y)$ on this compact pole-time range, while $R(y)r_{y}=r_{0}$. The Li--Yau's parabolic Harnack inequality \cite[Theorem~2.2(i), pp.~167--168]{LY1986}, with potential $q=0$, curvature parameter $K=0$, and the two spatial points both equal to $x$, gives
	\begin{equation*}
		H(x, y, s)\le \left(\frac{t}{s}\right)^{n/2}H(x, y, t) \le 2^{n/2}H(x, y, t), \qquad t-4r_{y}^2\le s\le t.
	\end{equation*}
	Therefore
	\begin{equation}\label{eq:cylinder-control}
		\sup_{\mathcal{C}_{x, y, t}}H(\cdot, y, \cdot) \le C H(x, y, t).
	\end{equation}

	It remains to convert \eqref{eq:cylinder-control} into derivative bounds. We use the standard local parabolic estimate on nested backward cylinders: after decreasing $r_{0}$ if necessary, for every $m,\ell \ge 0$ there exists $C_{K, m,\ell}$ such that any solution of $\partial_su=\Delta_gu$ on
	\[
		B(x, 2r)\times[t-4r^{2}, t], \qquad x\in K,\quad 0<r\le r_{0},
	\]
	satisfies
	\begin{equation}\label{eq:para_est}
		|\nabla^{m}\partial_{s}^\ell u|(x, t) \le C_{K, m, \ell}r^{-m-2\ell} \sup_{B(x, 2r)\times[t-4r^2,t]}|u|.
	\end{equation}
	This is the usual one-sided interior Schauder estimate followed by its higher-regularity bootstrap \cite[Theorems~8.11.1 and~8.12.1, p130--132]{Kry1996}; see also \cite[Chapter~IV, p47, and Theorem~4.9, p59--60]{Lie1996}. The estimate is uniform for $x\in K$ because a finite coordinate cover of a fixed neighborhood of $K$ has uniformly controlled smooth coefficients. Its terminal-time formulation uses only the lateral and earlier-time parabolic boundary, and hence requires no values after time $t$.

	Applying \eqref{eq:para_est} to $u(z, s)=H(z, y, s)$ with $r=r_{y}$, and then using \eqref{eq:cylinder-control}, gives
	\[
		|\nabla_{x}^{m} \partial_{t}^\ell H(x, y, t)| \le C_{K,m,\ell}r_{y}^{-m-2\ell}H(x, y, t).
	\]
	Since $r_{y}=r_{0}/R(y)$, this proves \eqref{eq:rel_heat}.
	
	Finally, Li--Yau's Gaussian upper estimate \cite[Corollary~3.1]{LY1986} gives, after choosing any $\varepsilon\in(0,1)$ and weakening the Gaussian exponent if necessary,
	\[
		H(x,y,t)\le \frac{C}{\Vol B(x,\sqrt t)} \exp\!\left(-\frac{d(x,y)^2}{5t}\right).
	\]
	Here $\inf_{x\in K}\Vol B(x,\sqrt a)>0$, $t\le b$, and $d(x, y)\ge d(o, y)-\sup_{x\in K}d(o, x)$; in particular, $d(x, y)^2\ge \frac12d(o, y)^2-C_K$. Hence
	\begin{equation*}
		H(x, y, t)\le C\exp\bigl(-c\,d(o,y)^2\bigr), \qquad (x, t)\in K\times[a, b].
	\end{equation*}
	The polynomial factor in \eqref{eq:heat_maj} is absorbed by weakening $c$. Bishop--Gromov also gives at most Euclidean volume growth, so this Gaussian is integrable. The last assertion follows from \eqref{eq:rel_heat} and \eqref{eq:heat_maj}.
\end{proof}

\begin{remark}
	We emphasize that the constants in \autoref{lem:heat-jets} are local regularity constants and may depend on the metric near the compact set $K$. They are used to justify positive-time smoothness and differentiation under the integral in $y$. None of these constants enters the dimension only constant in \autoref{thm:two-unit} or in the final volume constant in \autoref{thm:Main}. In particular not bounded geometry hypothesis is imposed.
\end{remark}

\subsection{Proof of \autoref{prop:HeatTensorRegularity}}
\label{proof:HeatTensorRegularity}

Fix $K\Subset M$ and $0<a<b<\infty$, and put $R(y)=1+d(o,y)$. By the product and quotient rules, every positive-order pole-space-time derivative of $h=-\log H$ is a polynomial in the relative jets $H^{-1}\nabla_{x}^{m}\partial_{t}^\ell H$. Hence every derivative of
\begin{equation*}
	 H d_{x}h\otimes d_{x}h
\end{equation*}
is bounded on $K\times[a,b]$ by $H(x,y,t)R(y)^{N}$ for some $N$. The majorant in \autoref{lem:heat-jets} is integrable in $y$, uniformly on this cylinder. Dominated differentiation first shows that $\mathsf{G}$ is finite and smooth. Its derivatives are therefore locally bounded coefficients. Applying the same product and quotient rules to
\begin{equation*}
	 H\left|\hes_{x}h-\frac{\mathsf{G}}{2t}\right|^{2}
\end{equation*}
now gives the same type of majorant and shows that $D$ is finite and smooth. Since
\begin{equation*}
 \mathsf{E}=\frac t2(n-\tr_g\mathsf{G}),
 \qquad \mathsf{Q}=4t^{2}D+|g-\mathsf{G}|^{2} +2t\langle\Ric,\mathsf{G}\rangle,
\end{equation*}
the fields $\mathsf{E}$ and $\mathsf{Q}$ are finite and smooth as well. The same majorants justify every pole-space-time differentiation under the output-variable integrals. They also show that every derivative of the Hilbert-valued tensors $\mathcal{A}_{t}$ and $\mathcal{B}_{t}$ introduced in the proof of \autoref{prop:SpecProj_Regularity} has squared $L^{2}_{y}$ norm bounded by $\int_{M}H(x,y,t)R(y)^N\dd\Vol_{y}$ for some $N$; hence those tensors are smooth as $L^{2}_{y}$-valued maps. This proves \autoref{prop:HeatTensorRegularity}.

Only derivatives of finite parabolic order are used in the main proof; the all-order statement costs no additional argument once the interior estimate in \autoref{lem:heat-jets} is available.

\clearpage
\printbibliography
\end{document}